\documentclass[a4paper,11pt]{article}
\usepackage{geometry}
\usepackage[english]{babel}
\usepackage[utf8]{inputenc}

\usepackage{mathpazo}
\usepackage{authblk}
\usepackage{amsfonts, amssymb, amsmath, amsthm}
\usepackage{amsbsy}
\usepackage{xfrac}
\usepackage{xcolor}

\usepackage{hyperref}

\usepackage{tabularx}

\usepackage{fancyhdr}
\usepackage{ifthen}

\newtheoremstyle{mystyle}{}{}{\rmfamily}%
{}{\normalfont\bfseries}{ }{ }{}

\hypersetup{colorlinks=true, linkcolor=blue, citecolor=blue, urlcolor=blue}

\newtheorem{theorem}{Theorem}[section]
\newtheorem{lemma}[theorem]{Lemma}
\newtheorem{definition}[theorem]{Definition}
\newtheorem*{claim}{Claim}
\newtheorem{corollary}[theorem]{Corollary}
\newtheorem{proposition}[theorem]{Proposition}
\newtheorem{remark}[theorem]{Remark}

\theoremstyle{mystyle}
\newtheorem{question}[theorem]{Question}
\newtheorem{example}[theorem]{\bf Example}

\title{
Quasi-rigid operators and  hyper-recurrence
}
\author[1]{Manuel Saavedra}
\author[2]{Manuel Stadlbauer}
\affil[1,2]{Instituto de Matemática, Universidade Federal do Rio de Janeiro, RJ, Brazil.}
\affil[1]{\texttt{saavmath@im.ufrj.br}}
\affil[2]{\texttt{manuel@im.ufrj.br}}

\date{}

\fancypagestyle{customstyle}{
\fancyhf{}
\fancyhead[L]{\ifthenelse{\isodd{\thepage}}{}{\thepage}}
\fancyhead[R]{\ifthenelse{\isodd{\thepage}}{\thepage}{}}
\fancyhead[CO]{\ifthenelse{\isodd{\thepage}}{Manuel Saavedra and Manuel Stadlbauer}{Quasi-rigid Operators and  hyper-recurrence}}
  
}

\usepackage{enumitem}
\setlist{itemsep=-1mm}
\setlist[enumerate]{label = (\arabic*),
ref = (\arabic*)}

\begin{document}
\maketitle

\setlength{\headheight}{14.49998pt}

\begin{abstract}
 \noindent We study recurrent operators from a new perspective by introducing the notion of hyper-recurrent operators and establish robust connections with quasi-rigid operators. For example, we prove that a recurrent operator on a separable Banach space is quasi-rigid if and only if it is a linear factor of a hyper-recurrent operator, and show that the quasi-rigid operators found in Costakis, Manoussos and Parissis's work, along with many others, are, in fact, hyper-recurrent operators. Furthermore, we provide a negative answer, using a class of operators introduced by Tapia, to the question by Costakis et al. whether  $T \oplus T$ is recurrent whenever $T$ is.


 \end{abstract}


\section{Introduction} In this note, we discuss recurrence properties of a continuous operator $T$ acting on a Fréchet or Banach space $X$ by analysing the asymptotic behaviour along subsequences $(\omega_n)$ such that $\lim_n T^{\omega_n}(x) = x$ for some $x\in X$. In order to do so, we begin with recalling the notions of \emph{recurrence}, \emph{rigidity} and \emph{quasi-rigidity}. That is, $T$ is referred to as recurrent if there exists a dense subset $\Lambda$ of $X$ such that each element of this set is an accumulation point of its forward orbit. This implies that for each $x$ in this dense subset, there exists
a subsequence $(\omega_n)$ such that $\lim_n T^{\omega_n}(x) = x$. Furthermore, if this subsequence can be chosen independent of $x \in \Lambda$, then the operator is called \emph{quasi-rigid}. Or in other words, there exists a subsequence $(\omega_n)$ such that $\lim_n T^{\omega_n}(x) = x$ for all $x \in \Lambda$. If $\Lambda =X$, then a quasi-rigid operator is referred to as \emph{rigid}.

We now change our point of view and consider, for a given strictly increasing sequence $\omega$, the set $\mathfrak{L}(\omega) := \{ x: \lim_n T^{\omega_n}(x) = x\}$. In analogy to hypercyclic operators, we introduce the notion of \emph{hyper-recurrent} operators as follows. We refer to $T$ as hyper-recurrent if there exists a recurrent vector $x\in X$ such that $x \in \mathfrak{L}(\omega)$ implies that $\overline{\mathfrak{L}(\omega)} = X$. Or in other words, recurrence in $x$ implies convergence to the identity on a dense set with respect to the same sequence.  As it easily can  be seen,
\[
\hbox{hypercyclic} \Rightarrow \hbox{hyper-recurrent} \Rightarrow \hbox{quasi-rigid}  \Rightarrow  \hbox{recurrent}.
\]
The aim of this note is to study hyper-recurrence and, in particular, analyse the gaps between recurrence, quasi-rigidity and hyper-recurrence. To start, we remark that the these implications are strict: for the first, it suffices to consider the identity operator whereas Propositions \ref{T2}, \ref{RnH} and \ref{prop:quasi-rigid and not hyper-recurrent} below provide negative answers to the remaining ones.

The recent interest in recurrent operators probably was initiated by the publications \cite{Cos, Cos2} by Costakis, Manoussos and Parissis. In there, the authors provide a systematic study of the spectral properties of recurrent and rigid operators and obtained characterizations of relevant classes of recurrent operators on classical Fréchet and Banach spaces like weighted shifts, composition operators, multiplication operators, power-bounded operators (for the definition, see Section \ref{sec:Preliminaries}) and operators on finite dimensional spaces. Furthermore, \cite{Cos} contains several open questions, which are now discussed by the community.

For example, Mazet and Saias recently constructed a class of rigid operators in \cite{Mazet}, which are either non-invertible or have a non-rigid inverse. In particular, this class provides a negative answer to \textcolor{black}{two} of the questions in \cite{Cos}. However, the main focus of research of the last decade was put on hypercyclic vectors and their relation to Devaney chaos, i.e. the density of periodic points. A recent breakthrough in this direction is due to Menet (\cite{Menet--Linear-Chaos-And-Frequent--TAMS-2017}) who showed that a Devaney chaotic operator is reiteratively hypercyclic but not necessarily frequently hypercyclic
(for the definitions, see  \cite{Menet--Linear-Chaos-And-Frequent--TAMS-2017}).

In here, our focus is on recurrence instead of hypercyclicity. Motivated by the successful analysis of hypercyclic operators in terms of frequent, U-frequent or reiterative hypercyclicity (see, e.g., \cite{Gri,Fhyp,Menet--Linear-Chaos-And-Frequent--TAMS-2017}), these concepts recently also were transferred to recurrent, but not necessarily hypercyclic operators. That is, one refers to an operator $T$ as \emph{reiteratively/ U-frequently/ frequently recurrent} if the set
$ \{ n  \in \mathbb{N}:  T^n x \in A \} $
has a strictly positive upper Banach density/ upper density/ lower-density (\cite{Cos2, Boni, Grivaux, Sophie}) for $x$ in a dense set and any neighbourhood $U$ of $x$. For example, Costakis, Manoussos and Parissis related U-frequent recurrence to multiple recurrence (\cite{Cos2}), and Grivaux and López-Martínez showed in \cite{Grivaux} that these recurrence properties coincide in case of a reflexive Banach space. Furthermore, the notion of subspace hypercyclicity, introduced by in \cite{Madore}, recently was carried to recurrent operators in \cite{Moos}.

However, as one might see from the definition of hyper-recurrence, our approach is different as we are asking not for rates of returns but for a qualitative, global behaviour. That is, if $x$ is hyper-recurrent and $T^{\omega_n}x \to x$, then $T^{\omega_n}$ converges to the identity on a dense set.

The paper is structured as follows. In Section \ref{sec:Preliminaries} we recall underlying concepts and collect some basic results on the set of recurrent points. In Section \ref{sec:quasi-rigid}, we then show that quasi-rigidity is equivalent to $\bigoplus_{i=1}^n T: X^n \to X^n$ being recurrent for any $n \in \mathbb{N}$ (Theorem \ref{equiv}) and use this characterization and a construction by Augé (\cite{Auge}) and Tapia (\cite{Tapia}) in order to construct a recurrent operator which is not rigid (Theorem \ref{ans-neg} and Proposition \ref{T2}). This, in particular, answers question 9.6 in \cite{Cos}. Here, we would like to point out that the same results was independently obtained by Grivaux, López-Martínez and Peris in the preprint \cite{Sophie} by a different method.

In Section \ref{sec:hyper}, we then, among other things, introduce the notion of hyper-recurrence, give examples of quasi-rigid operators which are not hyper-recurrent (Propositions \ref{RnH} and \ref{prop:quasi-rigid and not hyper-recurrent}) and show in Theorem \ref{theo:factor} that an operator on a Fréchet space is quasi-rigid if and only if it is a factor of a hyper-recurrent one. Section \ref{sec:exploring} then is devoted to a discussion of hyper-recurrence for certain classes of operators.
In particular, we show that recurrence implies  hyper-recurrence in the following cases.
\begin{enumerate}[label=(\alph*)]
	\item Power bounded operators on a separable complex Hilbert space (Theorem \ref{Hr-H}).
	\item Multiplication operators on a unital Banach algebra (Theorem \ref{multi}), Banach spaces of holomorphic functions (Theorem \ref{Banach-holo}), or $L^p(\nu)$ with respect to a $\sigma$-finite measure  (Theorem \ref{theo:mult}).
	\item Invertible operators with discrete spectrum acting on a {separable complex} Fréchet space  (Theorem \ref{disc-esp}), and operators with a dense set of periodic vectors on a Fréchet space (Theorem \ref{per-dense}).
	\item Composition operators on $\mathrm{H}(\mathbb{C})$, $\mathrm{H}(\mathbb{C}^{*})$ $\mathrm{H}(\mathbb{D})$, the Hardy space $\mathrm{H}^{2}(\mathbb{D})$, the Weighted Hardy space of entire functions and the Hardy space of Dirichlet series (Theorem \ref{composition}).
\end{enumerate}

\section{Preliminaries} \label{sec:Preliminaries}

We now introduce some basic notions of recurrence. In order to do so, fix a continuous operator $T$ on a Fréchet space $X$. Then $x \in X$ is referred to as  recurrent if there exists a strictly increasing sequence $(\omega_n)$ such that $\lim_{n \to \infty} T^{\omega_n}(x)=x$.  Furthermore, we refer to $\mathrm{Rec}(T)$ as the set of all recurrent vectors, and recall that $T$ is a recurrent operator if $\overline{\mathrm{Rec}(T)}=X$.
In order to study recurrent operators in more detail, we consider the collection $\mathfrak{C}$ of sequences in $\mathbb{N}$ defined by
\begin{eqnarray*}
\mathfrak{C} : = \left\{(\omega_{n})_{n\in \mathbb{N}}: \omega_{n}\uparrow \infty \hbox{ and there is }\; x\in X\setminus\{0\}\; \hbox{with}\; \lim_{n\to \infty} T^{\omega_{n}}x =  x
\right\},
\end{eqnarray*}
and observe that $\mathfrak{C}$ is always non-trivial if $T$ is recurrent. Also note that the shift map $\sigma: (\omega_1,\omega_2,\ldots) \mapsto (\omega_2,\omega_3,\ldots)$ acts on $\mathfrak{C}$.
For $\omega=(\omega_{n})_{n\in \mathbb{N}}\in \mathfrak{C}$, set
\begin{eqnarray} \label{def:limit-set-of-omega}
\mathfrak{L}(\omega) \equiv \mathfrak{L}_T(\omega):=\{x\in X: \lim_{n\rightarrow \infty} T^{\omega_{n}}x=x\}.
\end{eqnarray}

We now collect some immediate consequences of these definitions. Note that $T(\mathfrak{L}(\omega)) \subset \mathfrak{L}(\omega)$ and that $\mathrm{Rec}(T)=\bigcup_{\omega\in \mathfrak{C}}\mathfrak{L}(\omega)$. Moreover, $\mathfrak{L}(\omega)$ and $\mathrm{Rec}(T)$ are $T$-invariant subsets of $X$ and $T$ acts injectively on them. However, if $\mathrm{Rec}(T)=X$ then $T$ is not necessarily invertible (see \cite{Mazet}).

An important class of operators in the context of recurrent operators is defined as follows (for more details, see Subsection \ref{subsec:pbo} below).

\begin{definition} \label{def:power-bounded}
We refer to $T$ as \emph{power-bounded} if for all $x \in X$, the orbit $\mathcal{O}_x = \{T^n(x): n \in \mathbb{N}\}$ is van Neumann-bounded in $X$. Or, in other words, for each neighborhood $U$ of $0$, there exists a positive number $r$ such that for all $z \in \mathbb{C}$ with $|z|\geq r$, $\mathcal{O}_x \subset zU$.
\end{definition}
Furthermore, in the context of bounded operators on Banach spaces, $T$ is power-bounded if and only if $\sup_n \|T^n\| < \infty$.
We now collect some basic facts about recurrent operators (for more results on power-bounded operators, see Subsection \ref{subsec:pbo} below).

\begin{proposition}\label{esp}
Let $T$ be a recurrent operator on a Fréchet space $X$. Then $\mathfrak{L}(\omega)$  and $\overline{\mathfrak{L}(\omega)}$ are  $T$-invariant linear subspaces,  $\mathfrak{L}(\sigma(\omega))=\mathfrak{L}(\omega)$ and  $\mathfrak{L}(\omega) \subset \mathfrak{L}(\mu)$ for any subsequence $\mu$ of  $\omega\in \mathfrak{C}$.
 Moreover, if $T$ is power-bounded, then $\hbox{Rec}(T)=X$ and $\mathfrak{L}(\omega)$ is closed for any $\omega\in \mathfrak{C}$.
\end{proposition}

\begin{proof}We only prove the statements on power-bounded operators as the first three are immediate.
As the the proof of Lemma 3.1 in \cite{Cos} applies in verbatim to Fréchet spaces by choosing an invariant metric on $X$, one obtains that $\hbox{Rec}(T)$ always is closed. Hence, $\hbox{Rec}(T)=X$ by recurrence. In order to show that $\mathfrak{L}(\omega)$ is closed in $X$ for any $\omega  \in \mathfrak{C}$, it suffices to employ
the Banach-Steinhaus theorem for Fréchet spaces in order to obtain that $\{T^n\}_{n}$ is equicontinuous.
\end{proof}

Moreover, one should not expect that $\mathfrak{L}(\omega)$ has finite dimension.

\begin{proposition}\label{densidade} Let $T$ be a recurrent operator on an infinite-dimensional  Fréchet space $X$. Then
the set $\{x \in \mathrm{Rec}(T) :  \hbox{if }  x\in \mathfrak{L}(\omega) \hbox{ then }  \mathrm{dim}(\mathfrak{L}(\omega)) = \infty\}$
is dense in $X$.
\end{proposition}

\begin{proof}  For  $x \in X$, set $E_{x} :=  \mathrm{span}(\{T^{n}x : n \geq 0\})$. Now assume that $x \in \mathrm{Rec}(T)$ and  $\mathrm{dim}(E_x) < \infty$. As $E_x$ is a $T$-invariant subspace and $\mathrm{dim}(E_{x}) < \infty$, it follows from the structure theorem for recurrent operators on finite dimensional spaces
(Theorems 4.1 and 4.2 in \cite{Cos}) that $T|_{E_{x}}$ is conjugated to a unitary matrix.
For $y \in \mathrm{Rec}(T)$ with  $\mathrm{dim}(E_y) < \infty$, it then follows that $\{T^{n}(x+y) : n \geq 0\})$ is a  bounded subset of $E_x + E_y$ and that there exists a strictly increasing sequence $(\omega_n)$ such that $T^{\omega_n}(x+y)$ converges. As $T$ is power-bounded on $E_{x}+E_{y}$, $T|_{E_{x}+E_{y}}$ is an invertible isometry by Theorems 4.1 and 4.2 in \cite{Cos}. Hence,
\[   \lim_{n \to \infty}  \|T^{\omega_n -\omega_{n-1}}(x+y)  - x+y \|_2 = \lim_{n \to \infty}  \|T^{\omega_n}(x+y) -  T^{\omega_{n-1}}(x+y)\|_2 =0.\]
As one may assume without loss of generality that $(\omega_n -\omega_{n-1})$ is strictly increasing,  $x+y \in \mathrm{Rec}(T)$.
This implies that $A := \{x \in \mathrm{Rec}(T): \mathrm{dim}(E_{x}) < \infty \}$ is a $T$-invariant subspace of $X$.

\medskip
\noindent\emph{Case (i).} So assume that $A \neq X$ and that $\hbox{Rec}(T) \setminus A$ is not dense in $X$. However, this would imply that $A$ contains a neighborhood of $0$, which is absurd. Hence, if $A \neq X$, then $\hbox{Rec}(T) \setminus A$ is dense. Finally, if $x \in \mathfrak{L}(\omega) \setminus A$ for some $\omega \in \mathfrak{C}$, then  $\mathrm{dim}(\mathfrak{L}(\omega)) = \infty$ as $E_{x} \subset \mathfrak{L}(\omega)$.

\noindent\emph{Case (ii).}
If $A = X$, then $\mathrm{dim}(E_{x}) < \infty$ for each $x \in X$ and $\hbox{Rec}(T) = X$.
As $\mathrm{Rec}(T) = X$, it follows that $T$ is injective. We now fix $p \in X$. Then $T|_{E_{p}} : E_{p}\rightarrow E_{p}$ is an injective endomorphism of a finite dimensional subspace and hence invertible. Therefore, there exists $q\in E_{p}$ with $T(q)=p$. Thus, $T$ is invertible by the open mapping theorem.

Now choose $x \in X$ and a closed subspace $M_x$ such that $X = E_{x} \oplus M_{x}$. Then, as $M_x \cap T^{-1}(E_x) = M_x \cap E_x = \{0\}$, it follows that $M_x$ is $T$-invariant and that $E_x \cap E_y   = \{0\}$ for all $y \in M_x$.
We now proceed by induction as follows. Assume that $x_0:=x$, $\epsilon > 0$ and that $x_1, \dots , x_n \in X$ are chosen so that $E_{x_i} \cap E_{x_j} = \{ 0\}$ for $i \neq j$ and that $d(0,x_i) < \epsilon 2^{-i}$ for $1\leq i,j\leq n$. Then, as above, there exists a closed and $T$-invariant subspace $M_{n} \subset X$ with  $X = E_{x_1} \oplus \cdots \oplus E_{x_n} \oplus M_{n}$. Hence it suffices to choose $x_{n+1} \in M_{n}$ with $d(0,x_{n+1}) < \textcolor{blue}{\epsilon}2^{-n-1}$ in order to proceed by induction.

By the bounds on $d(0,x_i)$, it follows that  $z:= x + \sum_{i=1}^\infty x_i \in X$ and $d(x,z)< \epsilon$. Moreover, by construction, $\bigoplus_i E_{x_i}\subset \mathfrak{L}(\omega)$ for any $\omega$ with $z\in \mathfrak{L}(\omega)$. As $\epsilon > 0$ can be chosen arbitrarily, the statement is proven.
\end{proof}

\section{Quasi-rigid operatores} \label{sec:quasi-rigid}

The question whether the hypercyclicity of an operator $T$ implies that $T \oplus T$ is hypercyclic (cf. \cite{Herrero}) was intensively discussed in the scientific community due to its relation to weak mixing and the hypercyclicity criterion. In fact, it turned out that the three properties are equivalent  (\cite{Bes-Peris--Hereditarily-Hypercyclic-Operators--JFA-1999}) and that there are hypercyclic operators which do not satisfy this property (\cite{rosa}). This brings forth a natural question in the context of recurrent operators.

\begin{question}[Question 9.6, \cite{Cos}]
Let $T: X \rightarrow X$ be a recurrent operator on a separable Banach space $X$. Is it true that $T \oplus T$ is recurrent on $X \oplus X$?
\end{question}

In the remaining part of this section, we discuss possible solutions to this question. Recently, a negative answer based on a construction of Augé in \cite{Auge}
was abtained by Grivaux, López-Martínez and Peris (cf. \cite[Theorem 3.2]{Sophie}). However, we want to point that we independently arrived at the same conclusion in  Theorem \ref{ans-neg} and Proposition \ref{T2} through similar but different arguments, based on the Mycielski Theorem and the class of operators introduced by Tapia in
 \cite{Tapia}.

\begin{definition}\label{def:quasi-rigid}
We refer to the recurrent operator $T$ as quasi-rigid if there exists a sequence $\theta \in \mathfrak{C}$ with $\overline{\mathfrak{L}(\theta)}= X$. Moreover, if ${\mathfrak{L}(\theta)}= X$, then we refer to  $T$ as a rigid operator.
\end{definition}

As a first remark with respect to this definition, note that Proposition \ref{esp} implies that for any recurrent operator $T$ and $\omega \in \mathfrak{C}(T)$, the restriction $T: \overline{\mathfrak{L}(w)} \to \overline{\mathfrak{L}(w)}$ is quasi-rigid. If, in addition, $T$ is power-bounded, then the restriction is rigid.
In order to have a precise criterion for quasi-rigidity at hand, we study finite cartesian products of $T$. For ease of notation, we write $(X^{m}, T_{m})$ for the action of $\bigoplus_{n=1}^m T$ on the $m$-fold product space of $X$.

\begin{theorem}\label{equiv}
Consider a recurrent operator $T: X \rightarrow X$ on a separable Fréchet space $X$. The following assertions are equivalent.
\begin{enumerate}
    \item \label{item:theo-equiv-1}  $T$ is a quasi-rigid operator.
    \item \label{item:theo-equiv-2} For each $m \in \mathbb{N}$, $T_m: X^{m} \rightarrow X^{m}$ is recurrent.
\end{enumerate}
\end{theorem}

We now recall the statement of Mycielski's Theorem. A set  $\mathcal{K}$ is referred to as a  \emph{Mycielski set} if the intersection of $\mathcal{K}$ and any nonempty open set $U$ contains a Cantor set.

\begin{theorem}[Mycielski Theorem, Corollary 1.1 in \cite{Tan}]
Suppose that $X$ is a separable complete metric space without isolated points, and that for every $n\in \mathbb{N}$, the set $\mathcal{R}_{n}$ is residual in the product space $X^{n}$. Then
there is a Mycielski set $\mathcal{K}$ in $X$ such that
\begin{eqnarray*}
(x_{1},x_{2}, \cdots, x_{n})\in \mathcal{R}_{n}
\end{eqnarray*}
for each $n\in \mathbb{N}$ and any pairwise different $n$ points $x_{1}, x_{2}, \ldots, x_{n}$ in $\mathcal{K}$.
\end{theorem}

\begin{proof}[Proof of Theorem \ref{equiv}]
If the subspace $\mathfrak{L}(\theta)$ is dense in $X$ for some sequence $\theta \in \mathfrak{C}$, then $\mathfrak{L}(\theta)^m$ is dense in $X^m$. As $\mathfrak{L}(\theta)^m \subset \hbox{Rec}(T_m)$ for each positive integer $m$, we obtain that
 \ref{item:theo-equiv-1} implies  \ref{item:theo-equiv-2}.

Conversely, the recurrence of $T_m$ implies that $\hbox{Rec}(T_m)$ is a residual subset of $X^m$ for every $m \in \mathbb{N}$. Thus, by the Mycielski Theorem, there exists a dense set $\mathcal{K} \subset X$ with $\mathcal{K}^{m} \subset \hbox{Rec}(T_m)$ for all $m \in \mathbb{N}$. Through the separability of $X$, one obtains a countable dense set $\{y_{j}\}_{j\in\mathbb{N}} \subset \mathcal{K}$. For $k \in \mathbb{N}$,  set $H_{k}:=\bigcap_{j=1}^{k}N(y_{j},B(y_{j},1/k))$, with $N(a, A):=\{n\in \mathbb{N}: T^{n}(a)\in A\}$. As each of the $H_k$ is unbounded, there exists a strictly increasing sequence of positive integers {$\theta:=(\theta_{k})_{k}$} with {$\theta_{k}\in H_{k}$}. Hence $\mathfrak{L}(\theta)$ contains the sequence $\{y_{j}\}_{j\in \mathbb{N}}$. This implies that $\mathfrak{L}(\theta)$ is dense in $X$.
\end{proof}

\begin{remark} Our proof of Theorem \ref{equiv} is based on Mycielski's Theorem which gives rise to a  shorter proof of the result. It is worth noting that the  Mycielski Theorem also finds application within the domain of nonlinear dynamics, as illustrated in  \cite{Cai, Garcia, TanF}.
However, we want to stress here that Grivaux et al. (cf.  \cite[Theorem 2.5]{Sophie}) recently established Theorem \ref{equiv} in the context of Polish spaces through a  completely different proof, and that we only became aware of their result after already having proved Theorem \ref{equiv}.
\end{remark}

\subsection{Augé-Tapia operators} \label{subsec:tapia-operators}

By the Banach-Steinhaus Theorem for Banach spaces, the set of points with unbounded orbit is either empty or dense. Moreover, for finite-dimensional vector spaces, the stronger property holds that $A_{T}=\{x\in \mathbb{K}^{m}:\lim_{n\rightarrow \infty}\Vert{T^{n}x\Vert}=\infty\}$ is either empty or dense. Motivated by this fact, Prăjitură conjectured that this property might hold for general Banach spaces  (see \cite{Praji}). However, the conjecture was answered to the negative by Hájek and Smith in \cite{Hajek} in the same year. Thereafter, the techniques were extended and improved through works of Augé (\cite{Auge}) and Tapia (\cite{Tapia}) on wild dynamics.
\begin{definition}[\cite{Auge}]
    Let $T$ be a linear bounded operator on a Banach space $X$. We say that $T$ is a wild operator if $A_{T}$ and $\hbox{Rec}(T)$ have nonempty interior and form a partition of $X$, where the set $A_{T}$ is given by
    \begin{eqnarray*}
       A_{T}=\{x\in X:\lim_{n\rightarrow \infty}\Vert{T^{n}x\Vert}=\infty\}.
    \end{eqnarray*}
\end{definition}

The first examples of wild operators were obtained by Augé in \cite{Auge}, who showed that each infinite dimensional separable Banach space admits a wild operator through an explicit contruction based on the following object.
\begin{definition}[\cite{Auge}]
We say that a subset $F$ of $X$ is asymptotically separated if there exists a sequence $(g_{n})\subset X^{*}$ such that
$\liminf \vert{g_{n}(x)\vert}=0$ for all $x\in F$, and
$\lim \vert{g_{n}(x)\vert}=+\infty$ for all $x\notin F$.
\end{definition}

Note that, if $F$ is an asymptotically separated set, then it is automatically a $G_\delta$-set.
Moreover, as shown by Augé (\cite[Th. 1.1]{Auge}) and Tapia (\cite[Th. 3.1]{Tapia}), each asymptotically separated set gives rise to a bounded operator with the following property.
\begin{theorem}[Augé, Tapia]\label{Tapia-general}
 Let $X$ be a separable infinite dimensional complex Banach space, $V$ be a complemented, infinite codimensional subspace of $X$ and $\mathbb{P}$ a bounded projection onto $V$. Then there exists for any asymptotically separated set $F\subset V$ and a  bounded operator $T$ such that
$    A_{T}=\mathbb{P}^{-1}(V\setminus F)$ and  $\mathrm{Rec}(T)=\mathbb{P}^{-1}(F)$.
\end{theorem}
However, the  operators constructed by Augé are not recurrent as he put his focus on non-empty interiors of $A_{T}$ and $\mathrm{Rec}(T)$. This is where Tapia's work stands out by showing that any real or complex Banach space of dimension greater or equal to 2 admits an asymptotically separated set $U$ such that $\mathcal{U}$ and $\mathcal{U}^{c}$ are dense for the norm topology (Corollary 2.12 in \cite{Tapia}).
As an immediate consequence, one then obtains that every infinite dimensional separable complex space Banach space admits an operator $T$ such that $A_{T}$ and $\hbox{Rec}(T)$ form a partition of $X$ and both are dense (Cororally 3.3 in \cite{Tapia}).

\begin{definition}\label{Tapia-operator}
We refer to a bounded linear operator defined on a separable infinite-dimensional Banach space $T$ as
an \textcolor{black}{Augé-Tapia} operator of finite type $n$ if $T$ is as in Theorem \ref{Tapia-general}, where the space $V$ has dimension $n$ and $F$ is an asymptotically separated set  such that both $F$ and $F^c$ are dense in $V$.
\end{definition}
Furthermore, it almost immediately turns out that these operators form a class of operators that are recurrent but not quasi-rigid.

\begin{theorem}\label{ans-neg} A \textcolor{black}{Augé-Tapia} operator of finite type is recurrent but not quasi-rigid.
\end{theorem}

\begin{proof} Assume that $V$, $F$ and $\mathbb{P}$ are as in Theorem \ref{Tapia-general}, and that $T$ is quasi-rigid. Then there exists $\omega \in \mathfrak{C}$ such that $\mathfrak{L}(\omega)$ is dense in $X$. Hence, $\mathbb{P}(\mathfrak{L}(\omega))$ is dense in $V$, which then implies that $V = \mathbb{P}(\mathfrak{L}(\omega))$ by finite dimensionality. Hence,
\begin{eqnarray*}
    V=\mathbb{P}(\mathfrak{L}(\omega))\subset \mathbb{P}(\mathrm{Rec}(T))=F\subsetneq V.
\end{eqnarray*}
This proves the theorem.
\end{proof}

We now recall the construction by Augé. As $X$ is a separable Banach space, there exist by Theorem 1 in \cite{Ovsepian} sequences $(e_{n}, e^{*}_{n})_{n\in \mathbb{N}}\subset X\times X^{*}$ such that
    \begin{enumerate}
        \item $\hbox{span}\{e_{n}:n\in \mathbb{N}\}$ is dense in $X$,
        \item $e^{*}_{n}(e_{m})=\delta_{n,m}$,
        \item for each $n\in \mathbb{N}$, $\Vert{e_{n}\Vert}=1$ and $M:=\sup_{n\in \mathbb{N}}\Vert{e^{*}_{n}\Vert}<\infty$.
    \end{enumerate}
It then follows from the third property that $\mathbb{P}(x):= \sum_{i=1}^d e^{*}_{i}(x) e_{i}$ defines an operator from $X$ to
$V:= \mathrm{span}(e_{1},e_{2},\cdots, e_{d})$ and that  $\Vert{\mathbb{P}\Vert}\leq dM$.
Choose an asymptotically separated set $F \subset V$ with associated sequence $(g_n)$ in $X^{*}$, as well as an increasing sequence $(m_k)$ with
$m_{k}\vert m_{k+1}$   for all $n\in \mathbb{N}$ and
\[ \sum_{k\geq d+1}\tfrac{m_{k-2}}{m_{k-1}}\Vert{g_{k}\Vert} < \infty.\]
We now define an
operator $S$ based on these sequences as follows. With $\lambda_{k} := 1$ for $1 \leq k \leq d$  and  $\lambda_{k}=\exp \frac{i\pi}{m_{k}}$ for $k > d$, define
\[  S: \textstyle\sum_{j=1}^{\ell}x_{k}e_{k} \mapsto  \textstyle\sum_{k=1}^{\ell}\lambda_{k}x_{k}e_{k},  \]
acting on $\mathrm{span}(e_{1},e_{2},\cdots, e_{\ell})$,  for  $\ell \in  \mathbb{N}$. Observe that $S$ extends to a bounded operator, $S:X \to X$. We are now in position to define the operator introduced by Augé in \cite{Auge}.
For $x\in X$, let
\begin{eqnarray} \label{eq:auge-tapia-operator}
    Tx=Sx + \sum_{k=d+1}^{\infty}\frac{1}{{m}_{k-1}} g_{k}(\mathbb{P}x)e_{k}.
\end{eqnarray}
As it easily can be verified by induction (cf. Lemma 3.5 in \cite{Auge}), we have for any $x\in X$ and $n \in \mathbb{N}$ that
\begin{eqnarray*}
  T^{n}x=S^{n}x+\sum_{k=d+1}^{\infty}\frac{ \sum_{j=0}^{n-1} \lambda_{k}^{j} }{m_{k-1}}  g_{k}(\mathbb{P}x)e_{k} =
  S^{n}x+\sum_{k=d+1}^{\infty}\frac{\lambda_{k}^{n} -1}{m_{k-1} (\lambda_{k} -1)}  g_{k}(\mathbb{P}x)e_{k}.
\end{eqnarray*}

As shown by Augé, this continuous linear operator $T$ satisfies (cf. \cite[p. 2108,2109]{Auge}) that
$A_{T}=\mathbb{P}^{-1}(V\setminus F)$ and $\mathrm{Rec}(T)=\mathbb{P}^{-1}(F)$. Furthermore, it follows from the underlying estimates, for any sequence $\theta_n \uparrow \infty$, that
\begin{equation} \label{eq:convergence-to-0-vs-recurrence}
 { \lim_{n \to \infty} \| T^{2m_{\theta_{n}-1}}(x) - x\| =0 \iff \lim _{n \to \infty} g_{\theta_n}(\mathbb{P}x) =0.}
\end{equation}
In particular, if $F$ and $F^c$ are dense in $V$, $T$ is an Augé-Tapia operator of finite type $d$.
As a first application, we show that the $d$-th product of the above operator no longer is recurrent.

\begin{proposition}\label{T2} Let $X$ be a separable infinite dimensional complex Banach space and $T$ be the {Augé-Tapia} operator of finite type $d$
defined through \eqref{eq:auge-tapia-operator} with respect to an asymptotically separated set $F \subset \mathbb{C}^d$ such that $F$ and $F^c$ are dense.
Then $T_{d}$ is not recurrent.
\end{proposition}
\begin{proof}
Assume that $T_d$ is recurrent. Then the set  $\mathrm{Rec}(T_{d}) \subset   X^{d}$ is residual. Furthermore, as
\begin{eqnarray*}
X_0:=  X\setminus  \bigcup_{k\geq d+1}\left\{x\in X:e_{k}^{*}x=-\tfrac{g_{k}(\mathbb{P}x)}{(\lambda_{k}-1)m_{k-1}} \right\}  \subset  X
\end{eqnarray*}
also is residual, there exist $y_1, \ldots , y_d \in   X_{0}$ such that $(y_{1}, y_2, \ldots, y_d) \in \mathrm{Rec}(T_{d})$, and such that $\{\mathbb{P}(y_{i})\}_{i=1}^{d}$ is linearly independent in $\mathrm{span} (\{e_1,\ldots , e_d\})$. By recurrence, there is a sequence $\theta=(\theta_{\ell})_{\ell}\in \mathfrak{C}$ such that
$\lim_\ell  T^{\theta_{\ell}}(y_{i}) = y_i$
for each $i\in\{1,\ldots, n\}$.
By applying $e_{k}^{*}$ to the iterates of $T$, we obtain for $k\geq d+1$  that
\begin{eqnarray*}
    e^{*}_{k}(T^{\theta_\ell}(y_i) - y_i) & = & e^{*}_{k}    \left(S^{\theta_{\ell}}y_{i}+\sum_{k\geq d+1} \frac{\sum_{j=0}^{\theta_{\ell}-1}\lambda_{k}^{j} }{m_{k-1}}g_{k}(\mathbb{P}y_{i})e_{k}-y_{i}\right)\\
    & = & ({\lambda_{k}^{\theta_{\ell}}-1}) \left({e^{*}_{k}(y_{i})+ \frac{g_{k}(\mathbb{P}y_{i})}{(\lambda_{k}-1)m_{k-1}}}\right)
\end{eqnarray*}
Now observe that the second term in the above identity does not vanish as $y_i \in X_0$. Hence, by recurrence,  $\lambda_{k}^{\theta_{\ell}}$ has to converge  to $1$ as $\ell$ tends to infinity. \textcolor{black}{Hence,
$\mathrm{span}(e_{i}:i>d)\subset \mathfrak{L}(\theta)$ and $W_d \subset {\mathfrak{L}(\theta)}$, for $W_{d} := \mathrm{span} (\{y_i\})$. However, $\mathbb{P}(W_d) = \mathrm{span} (\{e_1,\ldots , e_d\})$, which implies that $\overline{\mathfrak{L}(\theta)}= X$. This contradicts Theorem \ref{ans-neg}.}
\end{proof}

\begin{remark} It is worth pointing out the differences to the results by Grivaux, López-Martínez and Peris in \cite{Sophie}. In there,
the authors gave an example of an operator $T$, also defined through \eqref{eq:auge-tapia-operator}, such that $T_{d-1}$ is recurrent and $T_{d}$ is not.
However, instead of using linear forms coming from an asymptotically separated set, they use a given sequence of linear forms $(\lambda_n g_n)$ such that $\{g_n\}$ is dense in the unit sphere and $\lambda_n$ growths in an appropriate way.
\end{remark}

\section{Hyper-recurrent operators} \label{sec:hyper}

In this section, we analyze a class of operators such that there exist a vector who, in rough terms, describes the complete recurrrent behaviour of $T$. Due to their similarity with hypercyclic operators, we will refer to them as \emph{hyper-recurrent} vectors (see Definition \ref{def:hyper-recurrent}).
Note that it follows immediately from the definitions that a hyper-recurrent operator is quasi-rigid. Furthermore, as shown below, quasi-rigid operators are factors of hyper-recurrent ones. As a first step, we introduce the concept of stationary sequences.

\subsection{Stationary sequences}

Recall that an operator on $X$ is quasi-rigid if $\overline{\mathfrak{L}(\omega)}=X$ for some $\omega \in \mathfrak{C}$. In particular, $\overline{\mathfrak{L}(\omega)}=\overline{\mathfrak{L}(\mu)}$ for each $\mu$ subsequence of $\omega$. However, in order to
study recurrent operators without this property, like, e.g. the Tapia operators (cf.  Theorem \ref{ans-neg}), we give the following definition.

\begin{definition}
The sequence $\omega\in\mathfrak{C}_{T}$ is stationary if
$\overline{\mathfrak{L}(\omega)}=\overline{\mathfrak{L}(\mu)}$
for each $\mu$ subsequence of $\omega$.
\end{definition}

We now show that stationary sequences always exist \textcolor{black}{on separable Fréchet spaces}.

\begin{theorem}\label{stati}
Let $T: X\rightarrow X$ be a recurrent operator in a separable Fréchet space. Then each sequence in $\mathfrak{C}$ admits stationary subsequence.
\end{theorem}

\begin{proof} For a fixed a sequence $\omega \in \mathfrak{C}$, define $\mathfrak{A}(\omega)$ as the set of all subsequences of $\omega$. We introduce a strict partial order on $\mathfrak{A}(\omega)$ given by

\textcolor{black}{\begin{eqnarray*}
\mu\prec \nu \hbox{ if } \;\; \exists n\geq 0\; \text{such that}\; \sigma^{n}(\nu)\; \text{is a subsequence of}\; \mu\; \text{with}\:  \overline{\mathfrak{L}(\mu)}\subsetneq \overline{\mathfrak{L}(\nu)}.
\end{eqnarray*}}

Our aim is to show that every totally ordered subset of $\mathfrak{A}(\omega)$ has an upper bound, thereby invoking Zorn's Lemma to establish the existence of a maximal element, which will be a stationary subsequence of $\omega$.

Assume $\{\mu_{i}\}_{i\in I}$ is a totally ordered subset. If there exists an upper bound $\mu$ within $\{\mu_{i}\}_{i\in I}$, we are done. Otherwise, using the separability of $X$, we can construct a sequence $\{\nu_{n}\}_{n\in\mathbb{N}}\subset \{\mu_{i}\}_{i\in I}$ such that $\overline{ \bigcup_{i\in I} \mathfrak{L}(\mu_{i}) }=\overline{ \bigcup_{n\in\mathbb{N}}\mathfrak{L}(\nu_{n})}$. It's clear that $\{\nu_{n}\}_{n\in\mathbb{N}}\subset\mathfrak{A}(\omega)$ is totally ordered. We now show that for each $\alpha\in\{\mu_{i}\}_{i\in I}$ there exists $n\in\mathbb{N}$ such that $\alpha\prec \nu_{n}$.
Suppose the opposite, i.e. that there exists $\beta\in \{\mu_{i}\}_{i\in I}$ with $\nu_{n}\prec \beta$ for each $n\in \mathbb{N}$. This implies that  $\overline{ \bigcup_{n\in\mathbb{N}}\mathfrak{L}(\nu_{n})}\subset \overline{\mathfrak{L}(\beta)}$, which, due to $\beta$ not being an upper bound, leads to the strict inclusion $\overline{\mathfrak{L}(\beta)}\subsetneq\overline{ \bigcup_{i\in I} \mathfrak{L}(\mu_{i})}$, a contradiction.

Hence, it remains to establish an upper bound for $\{\nu_{n}\}_{n\in\mathbb{N}}$. To do this, we first may assume without loss of generality that $\nu_n \prec \nu_{n+1}$ for all $n  \in \mathbb{N}$. As $\{\nu_{n}\}$ is totally ordered, there exists an increasing sequence $(m_k)$ such that $\sigma^{m_{k}}(\nu_{k})$ is subsequence of $\sigma^{m_{k-1}}(\nu_{k-1})$ for each $k\geq 2$. Now choose a sequence $(a_{k})_k$ such that $a_{k}\in \sigma^{m_{k}}(\nu_{k})$ and $a_{k}\uparrow \infty$ for all $k \in \mathbb{N}$.

For  $\psi = (a_{k})$, it then follows by construction that $\sigma^k(\psi)$ is a subsequence of $\nu_{k}$ and by Proposition \ref{esp}, that $\mathfrak{L}(\psi) \supset \mathfrak{L}(\nu_k)$ for all $k \in \mathbb{N}$. This proves that  each totally ordered subset of $(\mathfrak{A}(\omega), \prec)$ has an upper bound. Applying Zorn's Lemma, a maximal element $\psi$ exists. If $\gamma$ is a subsequence of $\psi$, then by maximality, $\overline{\mathfrak{L}(\psi)}=\overline{\mathfrak{L}(\gamma)}$, indicating that $\psi$ is a stationary subsequence of $\omega$.
\end{proof}

We remark that the above proof in verbatim applies to also to recurrent maps on Polish spaces. We now proceed with a dynamical characterizations of stationary sequences.
In order to do so, for $\omega \in \mathfrak{C}$ and $A$ open, set $N_\omega(A):= \{ n \in \mathbb{N} : T^{-n}(A) \cap A \neq 0 \}$.

\begin{proposition}
Let $X$ be a Fréchet space, $T: X\rightarrow X$ be a recurrent operator and $\omega=(\omega_{n})_{n\in \mathbb{N}}$ a stationary sequence for $T$.
\begin{enumerate}
 \item  We have $\overline{\mathfrak{L}(\omega)}=X$ if and only if $|N_\omega(A)|=\infty$  for each non-empty open set $A$.
 \item  If $T$ is power bounded, then ${\mathfrak{L}(\omega)} = \{ x: |N_\omega(B(x,\epsilon))| = \infty \hbox{ for all } \epsilon> 0 \}$.
\end{enumerate}
\end{proposition}

\begin{proof} We begin with the first assertion. So assume that $\overline{\mathfrak{L}(\omega)}=X$. Then, for any non-empty open set $A$ and $x \in A \cap \mathfrak{L}(\omega)$, it follows that  $T^{\omega_{n}}x \in A$ eventually. Hence,  $|N_\omega(A)|=\infty$.

On the other hand, if $|N_\omega(A)|=\infty$ for any non-empty and open set $A$,
then, for any $x\in X$, $\epsilon > 0$ and $N\in \mathbb{N}$, there exist
$y \in B(x,\epsilon)$, $\delta > 0$ and $n > N$ such that
\[ \overline{B(y,\delta))} \subset B(x,\epsilon) \hbox{ and }
  T^{n}( B(y,\delta))  \subset B(x, \epsilon). \]
Hence, by inductively applying this observation, one obtains sequences ${x_{n}}$ and ${\epsilon_{n}}$ such that $x_{n}\in B(x,\epsilon)$, $\epsilon_{n} \downarrow 0$, $\overline{B(x_{n+1}, \epsilon_{n+1})}\subset B(x_{n}, \epsilon_{n})$ and $T^{\omega_{k_{n}}}(B(x_{n+1}, \epsilon_{n+1}))  \subset B(x_{n},\epsilon_{n})$ for some subsequence $\nu:=(\omega_{k_{n}})$ of $(\omega_{n})$. By Kuratowski's theorem, $\{y\} :=\bigcap_{n} B(x_{n}, \epsilon_{n})$ is nonempty. Furthermore,
 $y \in \mathfrak{L}(\nu)$ as $\lim_{n \to \infty} T^{\omega_{k_{n}}}(y) =y$ by construction. Then, as $\omega$ is stationary, $X= \overline{\mathfrak{L}(\nu)} =  \overline{\mathfrak{L}(\omega)}$. Finally, as $x$ and $\epsilon$ can be chosen arbitrarily,
$\mathfrak{L}(\omega)$ is dense in $X$.

We now proceed to the second assertion.  Set $\Omega:=  \{ x \in X: |N_\omega(B(x,\epsilon))| = \infty \hbox{ for all } \epsilon> 0 \}$. Then $y \in \Omega^c$ if and only if there exists $\epsilon > 0$ such that $N_\omega(B(y,\epsilon))$ is a finite set.  As $N_\omega(A)$ is finite for any $A \subset B(y,\epsilon)$, it follows that $\Omega^c$ is open and that $B(y,\epsilon) \subset \mathfrak{L}(\omega)^c$. Hence, $\Omega$ is a closed set and $\overline{\mathfrak{L}(\omega)} \subset \Omega$

So it remains to show that ${\mathfrak{L}(\omega)} = \Omega$. For $x \in \Omega$, one can see that there exists a sequence $(x_{k})_{k}\subset X$ and a subsequence $\omega_{n_{k}}$ of $\omega$ such that $x_{k}\in B(x,1/k)$ and $T^{\omega_{n_{k}}}x_{k}\in B(x, 1/k)$. As $T$ is power bounded, this implies that
$ \lim_{k\to \infty} T^{\omega_{n_k}}(x) =  x$. Therefore, $x\in \mathfrak{L}(\omega)$ since $\omega$ is stationary
\end{proof}

The concept of mixing recurrence was introduced by Amouch, Lakrima and Jadida in \cite{Amouch-Lak} as follows. They refer to a dynamical system $T: X \rightarrow X$ as  \emph{mixing recurrent} if, for every nonempty open subset $U$ of $X$, there exists a positive integer $N$ such that $\displaystyle {T^{n}(U)\cap U\neq \emptyset\;\hbox{for all}\; n\geq N}$. It is evident that each topologically mixing operator is mixing recurrent. As an immediate consequence, we obtain the following
relation between quasi-rigidity and  stationarity.

\begin{corollary}\label{cor:mixing-recurrent-stationary}
Let $T$ be a mixing recurrent operator on $X$ and let $\omega \in \mathfrak{C}$.
Then $\omega$ is a stationary sequence if and only if $\overline{\mathfrak{L}(\omega)} =X$.
\end{corollary}

We now discuss stationary sequences from a dual point of view by considering those elements in $X$ whose recurrent sequences are automatically stationary. That is, we consider the set
\begin{eqnarray*}
   \mathrm{St}(T):=\{x \in \mathrm{Rec}(T): \hbox{if }  x \in \mathfrak{L}(\omega) \hbox{ for some } \omega,\hbox{then } \omega\hbox{ is a stationary sequence}\}.
\end{eqnarray*}

\begin{proposition}\label{estacionario}
Let $T\in \mathcal{L}(X)$ be a recurrent operator of a Fréchet space $X$ and assume that $x \in \mathrm{Rec}(T)$. Then $x\in \mathrm{St}(T)$ if and only if
there exists a closed $T$-invariant subspace $Y$ of $X$ such that $Y=\overline{\mathfrak{L}(\omega)}$ for all $\omega \in \mathfrak{C}$ with $x \in  \mathfrak{L}(\omega)$.
\end{proposition}

\begin{proof}
Set $\mathcal{Z} := \{ \omega \in \mathfrak{C} : x \in {\mathfrak{L}(\omega)}\}$ and assume that $\omega,\theta \in \mathcal{Z}$.
Then, with $\omega \cup \theta$ referring to the strictly increasing sequence which contains the elements of $\omega$ and $\theta$, it follows that
$x \in \mathfrak{L}(\omega \cup \theta)$. Hence, if $x \in \mathrm{St}(T)$, then
$\omega \cup \theta$ is stationary and, in particular, $
\overline{\mathfrak{L}(\theta\cup \omega)}=\overline{\mathfrak{L}(\omega)}= \overline{\mathfrak{L}(\theta)}$.
Hence, the first direction follows for  $Y:=\overline{\mathfrak{L}(\theta)}$, which is $T$-invariant by Proposition \ref{esp}.
The other direction is immediate.
\end{proof}

It is natural to ask whether  $\mathrm{St}(T)$ is non-empty. 

\begin{proposition}\label{RnH}
Let $X$ be a non-separable complex Hilbert space. Then there exists a rigid operator $T:X \to X$ such that $\mathrm{St}(T) = \emptyset$.
\end{proposition}

\begin{proof}
Assume that $\psi:=(a_{n})_{n}$ is a sequence such that $G=\{\lambda\in \mathbb{T}: \lambda^{a_{n}}\longrightarrow 1\}$ forms an uncountable group. Without loss of generality, we can assume that the Hilbert space is $\ell^{2}(G)$. So assume that $\{x_{\lambda}\}_{\lambda\in G}$ is an orthonormal basis and that $T$ is defined through
\begin{eqnarray*}
T\left(\sum_{\lambda\in G}c_{\lambda}x_{\lambda}\right) :=\sum_{\lambda\in G}\lambda c_{\lambda} x_{\lambda}.
\end{eqnarray*}
Then $T$ is a surjective isometry with $\hbox{span}\{x_{\lambda}:\lambda\in G\}\subset \mathfrak{L}(\psi)$. Hence $T$ is rigid by Proposition \ref{esp}.

We now make use of the minimality of the torus rotation (see, for example, \cite{Viana}). In particular, we will refer to a subset $M$ of $\mathbb{T}$ as log-rationally independent if for any finite subset $\{\lambda_1 ,\ldots \lambda_k\} \subset M$ and $n_1, \ldots, n_k \in \mathbb{Z}$, the identity $\prod_i \lambda_i^{ n_i} = 1$ implies that $n_1= \cdots = n_k =0$ (see also Proposition \ref{ortho}).
So choose a sequence $(\lambda_{k})_{k\in \mathbb{N}}$ in $G$ and assume that $x \in \overline{\mathrm{span}(\{x_{\lambda_k}\})}$. Then, as $G$ is uncountable, there exists $\lambda\in G$ such that $\lambda$ is not a root of unity and such that $\lambda$ is log-rationally independent from any finite subset of $\{\lambda_k\}$. Hence, for any $k \in \mathbb{N}$, there exists $I_k \subset \{1,\ldots, k\}$  such that $\{\lambda_i : i \in I_k\} \cup \{\lambda \} $ is log-rationally independent and that each $\lambda_j$ with $j \leq k$, $j \notin I_k$ is rationally dependent from $\{\lambda_i : i \in I_k\} \cup \{\lambda \} $.

Now fix $\beta \in  \mathbb{T}\textstyle{\setminus\{1\}}$. As log-rationally independence implies that the torus rotation is minimal and the remaining $\lambda_i$ can be written as rational combinations of the elements in $\{\lambda_i : i \in I_k\} \cup \{\lambda \}$. Hence, by minimality, there exists $\theta_k$ arbitrary large such that
\[ |\lambda^{\theta_k} - \beta| < 1/k, \quad |\lambda_i^{\theta_k} - 1| < 1/k  \; \forall i = 1, \ldots, k.\]
Hence, there exists a strictly increasing sequence $(\theta_{k})$
with $\lim_k \lambda^{\theta_k} = \beta$ and $\lim_k \lambda_i^{\theta_k} = 1 $ for all $i \geq 1$.
In particular,   $x_{\lambda} \notin \mathfrak{L}(\theta)$. Hence, $\mathfrak{L}(\theta) \neq \mathfrak{L}(\psi)$ which implies that $x \notin \mathrm{St}(T)$ by Proposition \ref{estacionario}.
\end{proof}

\subsection{Hyper-recurrent operators}
We now introduce the class of hyper-recurrent operators. For a recurrent operator $T$, set
\begin{eqnarray*}\label{hyper-sta}
     \mathrm{Hr}(T):=\{x \in \mathrm{Rec}(T): \hbox{if }  x \in \mathfrak{L}(\omega) \hbox{ for some } \omega,\hbox{then } \overline{\mathfrak{L}(\omega)} =X\}.
\end{eqnarray*}
In order to give an example, it suffices to consider hypercyclic operators. In this case, each hypercyclic vector is an element of $\mathrm{Hr}(T)$.
\begin{definition}\label{def:hyper-recurrent}
Let $T: X\rightarrow X$ be a recurrent operator on a Fréchet space $X$. A recurrent vector $x\in X$ is {hyper-recurrent} if for every $\omega \in \mathfrak{C}$ with $x\in \mathfrak{L}(\omega)$ we have that $\overline{\mathfrak{L}(\omega)}=X$. The set of hyper-recurrent vectors is denoted by $\mathrm{Hr}(T)$ and we say that $T$ is a hyper-recurrent operator if there exists a  hyper-recurrent vector for $T$.
\end{definition}

In other words, a recurrent vector $x$ is hyper-recurrent if $\lim_n T^{a_n} (x) = x$ for some sequence implies that $T^{a_n}$ converges pointwise to the identity on a dense set in $X$. As a consequence, if $T$ is a hyper-recurrent operator then $T$ is quasi-rigid.
Furthermore, it follows from Proposition \ref{estacionario} for $T$ with $\mathrm{St}(T) \neq \emptyset$ that
the restriction of $T$ to $Y$ as given in there is hyper-recurrent.

Furthermore, hypercyclic operators are by far not the only examples of hyper-recurrent operators as shown in Section \ref{sec:exploring} below. However, by Proposition \ref{RnH}, there exist rigid operators which are not hyper-recurrent.
We now collect some basic properties of hyper-recurrent operators.

\begin{proposition}
    If $T$ is hyper-recurrent then $T^{m}$ and $T_{m}$ are also hyper-recurrent for every positive integer $m$.
\end{proposition}

\begin{proof}
Let $x$ be a hyper-recurrent vector for $T$. According to \cite[Proposition 2.3]{Cos}, $x$ is recurrent for $T^{m}$ for each positive integer $m$. Now, consider a fixed positive integer $m$. If $(T^{m})^{\omega_{n}}x$ converges to $x$ for some sequence $\omega:=(\omega_{n})_{n}$, then $T^{m\omega_{n}}x$ converges to $x$. Since $x$ is hyper-recurrent for $T$, we can conclude that $(T^{m})^{\omega_{n}}$ converges pointwise to the identity operator on a dense set in $X$.

On the other hand, it is clear that the $m-$tupel $(x,0, \ldots, 0)$ is recurrent for $T_{m}$. Furthermore, if $T_{m}^{\omega_{n}}(x,0,\ldots,0)$ converges to $(x,0, \ldots,0)$, then $T^{\omega_{n}}$ converges pointwise to the identity on a dense set $X_{0} \subset X$ by hyper-recurrence of $T$. It is then clear that $T_{m}^{\omega_{n}}$ converges to the identity on the dense set $X_{0}^{m}$. Hence, $T_{m}$ is hyper-recurrent.
\end{proof}

From Corollary \ref{cor:mixing-recurrent-stationary}, one immediately obtains the following for mixing recurrent operators, using the fact that  $\mathrm{St}(T) \neq \emptyset$ if $T$ is topologically mixing.

\begin{proposition}
Let $X$ be a Fréchet space and $T\in \mathcal{L}(X)$ be a recurrent operator on $X$. If $T$ is mixing recurrent, then
$\mathrm{St}(T)=\mathrm{Hr}(T)$.
In particular, if $T$ is topologically mixing then $\mathrm{St}(T)=\mathrm{Hr}(T)$ is dense in $X$.
\end{proposition}

On the other hand, in case of an \textcolor{black}{Augé-Tapia} operator, $\mathrm{St}(T)$ and $\mathrm{Hr}(T)$ might differ. In here, $Y_{q}$ stands for the invariant set given by Proposition \ref{estacionario}, for $q\in  \mathrm{St}(T)$.

\begin{theorem}\label{Hr-sta}
Let $X$ be a separable infinite-dimensional complex Banach space. Then, there exists a recurrent operator $T$ acting on $X$ such that $\{q\in \textit{St}(T): \mathrm{dim}(X/Y_{q})=1\}$
is dense in $X$, and the set $Hr(T)$ is empty.
\end{theorem}

\begin{proof} We will show that the \textcolor{black}{Augé-Tapia} operator of Proposition \ref{T2} for $d=2$ satisfies this property. That is, we consider the operator defined through
\begin{eqnarray*}
     Tx=Sx+\sum_{k=3}^{\infty}\frac{1}{m_{k-1}} g_{k}(\mathbb{P}x)e_{k},
\end{eqnarray*}
where $\mathbb{P}: X \to V:= \hbox{span}(\{e_{1}, e_{2}\})$ is the canonical projection, $(g_k)$ is the sequence in $V^\ast$ associated with an asymptotically separated set $F \subset V$ and the $m_{k}$ are as above. Furthermore, without loss of generality, we assume that $g_k \neq 0$ for all $k \in \mathbb{N}$.

Now choose a countable and dense set $G \subset V$ contained within the residual set defined by
\begin{eqnarray*}
\displaystyle{ G\subset F \cap \left(V\setminus \left(\bigcup_{k=1}^{\infty} \mathrm{ker}(g_{k})\right)\right)}.
\end{eqnarray*}
In analogy to the proof of Proposition \ref{T2}, we now consider the set
\begin{eqnarray*}
X_{0}=\left\{p+\textstyle\sum_{j=3}^{\ell}\alpha_{j}e_{j}: p\in G, \ell \geq 3, \alpha_{k} \neq -\frac{g_{k}(p)}{(\lambda_{k}-1)m_{k-1}} \, \forall k \geq 3 \right\}.
\end{eqnarray*}
Observe that $X_0$ is dense in $X$ and that $X_0 \subset \mathrm{Rec}(T)$ by construction of the $m_k$. We will now show that $X_{0}\subset \mathrm{St}(T)$ and $\hbox{dim}(X/Y_{q})=1$ for each $q$ that belongs to $X_{0}$. In order to do so, fix $ q =  p+\sum_{j=3}^{\ell}\alpha_{j}e_{j} \in X_0$.  If $q\in \mathfrak{L}(\omega)$ for some $\omega:=(\omega_{n})_{n}$, then
\begin{eqnarray*}
T^{\omega_{n}}(q) - q =
\sum_{k=3}^{\ell} (\lambda_{k}^{\omega_{n}}-1) \left( \alpha_{k}  + \frac{g_{k}(p) }{ (\lambda_k - 1) m_{k-1}}\right) e_{k}    +
\sum_{k=\ell+1}^{\infty} \frac{(\lambda_{k}^{\omega_{n}}-1) g_{k}(p) }{ (\lambda_k - 1) m_{k-1}} e_k  \xrightarrow{n\rightarrow \infty} 0
\end{eqnarray*}
By applying $e_{k}^{*}$ for $3 \leq k \leq \ell$, this implies as in the proof of Proposition \ref{T2} through the construction of $X_0$, that  $\lambda_{k}^{\omega_{n}}$ has to converge to $1$. On the other hand, for $k > l$, it follows that
\[  \frac{(\lambda_{k}^{\omega_{n}}-1) g_{k}(p) }{ (\lambda_k - 1) m_{k-1}}   \xrightarrow{n\rightarrow \infty} 0. \]
Hence, $\lambda_{k}^{\omega_{n}}$ converges to $1$ as $n\rightarrow \infty$ for each $k\geq 3$.
In particular, as $T^{\omega_{n}}(e_k)  = \lambda_{k}^{\omega_{n}} e_k$ for each $k\geq 3$, it follows that $W:=\mathrm{span}(e_{k}:k\geq 3) \subset \mathfrak{L}(\omega)$.  Hence,
$\overline{\mathfrak{L}(\omega)}=\text{span}(p)\oplus \overline{W}$ by Theorem \ref{ans-neg}. This is, $Y_{q}=\text{span}(p)\oplus \overline{W}$ and $\text{dim}(X/Y_{q})=1$.
\end{proof}

A further remarkable fact is that hyper-recurrent operators are not so far from beeing quasi-rigid, as shown in the next result.

\begin{theorem} \label{theo:factor}
For  a recurrent operator  $T$ on a separable Fréchet or Banach space $X$, the following assertions are equivalent.
\begin{enumerate}
    \item The operator $T$ is quasi-rigid.
    \item The operator $T$ is a linear and continuous factor of a hyper-recurrent operator $S$ acting on a separable Fréchet or Banach space such that $\mathrm{Hr}(S)$ is dense in $Y$.
\end{enumerate}
\end{theorem}

\begin{proof}
We will give the proof within the framework of a separable Fréchet space. We start with an explicit construction of a hyper-recurrent operator which has a given quasi-rigid operator $T$ as a factor. Choose $\omega \in \mathfrak{C}$ such that $\mathfrak{L}_T(\omega)$ is dense in $X$. Due to the separability of $X$, there exists a countable set $\{q_{k}\}_{k\in \mathbb{N}} \subset \mathfrak{L}(\omega)$ which is dense in  $X$.
We now consider the  separable Fréchet space
\[\ell_{1}(X)={\left\{(x_{m})_{m\in \mathbb{N}}: x_{m}\in X\; \forall m\in \mathbb{N},  \; \textstyle\sum_{m\in \mathbb{N}}d(x_{m},0)<\infty \right\}},\]
equipped with the metric $d_1((x_m)_m,(y_m)_m) = \sum_m d(x_m,y_m)$,
and the operator $S$ defined through $S((x_{m})_{m})=(Tx_{m})_{m}$. As $T$ is quasi-rigid, $S$ is a recurrent operator on $\ell_{1}(X)$.
Furthermore, the projection $\Theta: \ell_{1}(X)\rightarrow X$ onto the first coordinate is continuous and satisfies $\Theta\circ S=T\circ \Theta$.

We proceed with the proof that $\mathrm{Hr}(S)$ is dense in $\ell_{1}(X)$. To achieve this, consider an arbitrary positive integer $\ell > 1$ and $\epsilon > 0$. For a vector $\overline{x} = (x_{1}, \ldots, x_{\ell}, 0, \ldots, 0, \ldots)$ in $\ell_{1}(X)$, there are $q_{k_{1}}, \ldots,q_{k_{\ell}}$ such that
$\sum_{i=1}^{\ell}d(x_{i}, q_{k_{i}})<\epsilon/3$. It is evident that for each $k$, the set $\{T^{\omega_{n}}q_{k}\}_{n}\cup \{q_{k}\}$ is bounded in $X$. Hence, there exist $r_{k}\in (0,1)$ such that $\{r_{k}T^{\omega_{n}}q_{k}\}_{n}\cup \{r_{k}q_{k}\}\subset B(0, \epsilon/2^{k+1})$. Set
\[
z:=(q_{k_{1}}, \ldots, q_{k_{\ell}}, r_{1}q_{1}, r_{2}q_{2}, r_{3}q_{3}, \ldots, r_{k}q_{k}, \ldots)\in \ell_{1}(X)
\]
One can see that $d(\overline{x}, z)<\epsilon$, and $z\in \mathfrak{L}_{S}(\omega)$. We claim that $z$ is a hyper-recurrent vector for $S$. If $z\in \mathfrak{L}_{S}(\theta)$ for some $\theta:=(\theta_{n})_{n}$, then
\begin{eqnarray*}
	d_1(S^{\theta_{n}}z, z)\geq \sum_{k=1}^{\infty} d(r_{k}(T^{\theta_{n}}q_{k}-q_{k}),0)\xrightarrow[n\to \infty]{} 0
\end{eqnarray*}
As consequence, $\{q_{k}\}_{k}\subset \mathfrak{L}_{T}(\theta)$. Therefore, $\overline{\mathfrak{L}_{S}(\theta)}=X$.
To prove the remaining direction, it suffices to note that a factor of a quasi-rigid operator is always quasi-rigid.
\end{proof}
In order to study the gap between quasi-rigidity and hyper-recurrence, we introduce the notation
\begin{center}
    $\eta(T):=\min \{\vert{A\vert}: A\subset \mathrm{Rec}(T)\; \text{is a countable set, such that if}\; A\subset \mathfrak{L}(\omega)\;\text{for some}\; \omega,\; \text{then}\; \overline{\mathfrak{L}(\omega)}=X\}$.
\end{center}
Observe that the quantity $\eta(T)$ can be perceived as the minimal number of vectors which are necessary to capture the essence of quasi-rigidity. For example, an operator $T$ is hyper-recurrent if and only if $\eta(T) = 1$. Moreover, assuming $n > 1$, it follows that $\eta(T) = n$ if and only if $T_{n}$ is hyper-recurrent while $T_{n-1}$ is not.

\begin{proposition} \label{prop:eta}
Let $X$ be a separable Banach space and $T\in \mathcal{L}(X)$ be a quasi-rigid operator. Then the following holds.
\begin{enumerate}
\item $\eta(T)=2$ if and only if, $T_2$ is hyper-recurrent and for every $(p,q)\in \mathrm{Hr}(T_{2})$ and every $r\in(0, \infty)$, there exists a strictly increasing sequence of positive integers $(a_{\ell})_{\ell\in \mathbb{N}}$ such that
\begin{eqnarray*}
    \lim_{\ell\rightarrow \infty}\frac{\Vert{T^{a_{\ell}}p-p\Vert}}{\Vert{T^{a_{\ell}}q-q\Vert}}=r.
\end{eqnarray*}
\item $\eta(T)=\infty$ if and only if, for every $n$ and for every $(q_{1}, \cdots, q_{n})\in \mathrm{Rec}(T_{n})$, there exists a sequence $\omega\in \mathfrak{C}$ such that $\mathfrak{L}(\omega)$ is not dense in $X$ and $\{q_{i}: i=1,\ldots, n\}\subset \mathfrak{L}(\omega)$.
\end{enumerate}
\end{proposition}

\begin{proof} We begin with the proof of the first statement.
If $\eta(T) = 2$, then $T_2$ is always hyperrecurrent. So assume that $(p, q) \in \mathrm{Hr}(T_{2})$. Then $p$ and $q$ are not periodic. Furthermore, suppose that there exists  $r > 0$ such that $r \notin A(p, q)$, with
\begin{eqnarray*}
    A(p,q):= \left\{t\in(0,\infty): \exists\; (a_{\ell})_{\ell}\; \hbox{such that}\; \lim_{\ell\rightarrow \infty}\frac{\Vert{T^{a_{\ell}}p-p\Vert}}{\Vert{T^{a_{\ell}}q-q\Vert}}=t \right\}.
\end{eqnarray*}
It is then clear that $p+rq$ is a recurrent vector for $T$. Moreover, if $p + rq$ lies in $\mathfrak{L}(\omega)$ for some $\omega \in \mathfrak{C}$, then
\begin{eqnarray*}
 \Vert{T^{\omega_{n}}q-q\Vert}\left\vert{r-\frac{\Vert{T^{\omega_n}p-p\Vert}}{\Vert{T^{\omega_n}q-q\Vert}}}\right\vert \leq \Vert{T^{\omega_{n}}(p+rq)-(p+rq)\Vert}  \xrightarrow[n\rightarrow \infty]{}0.
\end{eqnarray*}
Hence, $\lim_n \Vert{T^{\omega_{n}}q-q\Vert} =0$ as the second term does not converge to $0$ by assumption. Therefore, $\lim_n \Vert{T^{\omega_{n}}p-p\Vert} =0$ which then implies that  $(p,q)$ belongs to $\mathfrak{L}_{T_{2}}(\omega)$ and that  $\mathfrak{L}_{T_{2}}(\omega)$ is dense in $X^2$. So we obtain that  $p + rq \in \mathrm{Hr}(T)$ and $\eta(T)=1$, which is a contradiction.
To prove the converse, it remains to show that $T$ is not hyper-recurrent, or equivalently, that
$\mathrm{Hr}(T_{2})\cap \{(x,x): x\in X\}=\emptyset$. However, this follows from the condition on the quotients.

The second, remaining statement is an immediate consequence of the definition.
\end{proof}

Recall that Proposition \ref{RnH} implies that there exist rigid, non-hyper-recurrent operators on non-separable Hilbert spaces. In order to answer the question whether there are quasi-rigid operators which are not hyper-recurrent defined on a separable  Hilbert space, we adapt a construction in \cite{Sophie} in order to show the following.

\begin{proposition} \label{prop:quasi-rigid and not hyper-recurrent}
Let $X$ be a separable and infinite dimensional complex Hilbert space. Then there exists a rigid operator $T:X \to X$ with $\eta(T)=\infty$ and $\mathrm{St}(T) = \emptyset$. In particular, $T$ is not hyper-recurrent.
\end{proposition}

\begin{proof} Fix an orthonormal basis $\{ e_k : k \in \mathbb{Z}\}$ of $X$, set $E:= \mathrm{span}(\{e_k : k< 0\})$, $S:= \{ x \in E : \|x\| =1\}$ and let $\mathbb{P}: X \to E$ refer to the canonical projection. In order to define the operator, we now make use of a sequence $(\omega_k : k \geq 0)$ in $S$ and an increasing sequence $(m_k:k \geq 0)$ in $\mathbb{N}$ such that the following holds.
\begin{enumerate}
 \item $\{\omega_k\}$ is dense in $S$,
 \item $m_k |m_{k+1}$ for all $k \geq 0$ and $\lim_{j\to \infty} m_j \sum_{k> j} m_k^{-1} =0$.
\end{enumerate}
Let $\lambda_k := e^{\frac{2\pi i}{m_k}}$. The operator $T$ is now defined by, for $x = \sum_{k \in \mathbb{Z}} x_k e_k$,
\[
Tx = \sum_{k < 0} x_k e_k + \sum_{k \geq 0} \lambda_k x_k e_k + \sum_{k \geq 0} \frac{1}{m_{k-1}} \langle \omega_k, \mathbb{P}(x) \rangle e_k.
\]
It then follows in analogy to the Augé-Tapia operators that (see Fact  3.3.1 in \cite{Sophie} for our setting) that
\[
T^n(x) - x  = \sum_{k\geq 0} (\lambda_k^n -1)x_k e_k  + \sum_{k\geq 0} \frac{\lambda_k^n -1}{(\lambda_k -1)m_{k-1}} \langle \omega_k, \mathbb{P}(x) \rangle e_k.
\]
For $n = m_\ell$, the property that $m_{k+1}$ is a multiple of $m_{k}$ implies that
\begin{align*}
	T^{m_\ell}(x) - x =  & \sum_{k> \ell} (\lambda_k^{m_\ell} -1)x_k e_k  +  \frac{\lambda_{\ell + 1}^{m_\ell} -1}{(\lambda_{\ell + 1} -1)m_{\ell}} \langle \omega_{\ell+1}, \mathbb{P}(x) \rangle e_{\ell+1}
 \\ & +  \sum_{k> \ell + 1} \frac{\lambda_k^{m_\ell} -1}{(\lambda_k -1)m_{k-1}} \langle \omega_k, \mathbb{P}(x) \rangle e_k = I_1 + I_2 + I_3.
\end{align*}
It is now not hard to show that $I_1$ and $I_3$ tend to $0$ as $\ell \to \infty$ (cf. Fact 3.3.3 and the proof of Proposition 3.4 in \cite{Sophie}). Furthermore,
\[ \lim_{\ell \to \infty} \frac{\lambda_{\ell + 1}^{m_\ell} -1}{(\lambda_{\ell + 1} -1)m_{\ell}} =1. \]
Hence, as in the case of the Augé-Tapia operators, the recurrence of $x$ along $(m_\ell)$ is given by the asymptotics of $\langle \omega_k, \mathbb{P}(x) \rangle$.

 We now use this observation in order to show that $T$ is rigid. So choose a sequence $(k_\ell)_\ell$ such that ${\lim_{\ell \to \infty} \|\omega_{k_{\ell}+1}-e_{\ell}}\| =0$. Then, for each $x \in X$,
\begin{align*}
|\langle \omega_{k_{\ell}+1}, \mathbb{P}(x) \rangle |
	& \leq   | \langle e_\ell, \mathbb{P}(x) \rangle| +   |\langle \omega_{k_{\ell}+1} - e_\ell , \mathbb{P}(x) \rangle | \\
	& \stackrel{(\star)}{\leq}  | \langle e_\ell, \mathbb{P}(x) \rangle | +  \| \omega_{k_{\ell}+1} - e_\ell  \|  \| \mathbb{P}(x) \|\xrightarrow[\ell\rightarrow \infty]{\ast} 0,
\end{align*}
where we have used the Cauchy-Schwarz inequality in $(\star)$ and the Riemann-Lebesgue Lemma in $(\ast)$. Hence, $\lim_\ell T^{m_{k_l}}(x)=x$. As $x$ can be arbitrarily chosen, $T$ is rigid.

In order to show that $\eta(T) = \infty$, we now construct a further sequence $(k_\ell)_\ell$ for a given $n$-tupel $q_1, \ldots q_n \in X$ as follows. Choose  $\omega \in S$ in the orthogonal complement of  $\mathrm{span}(\{\mathbb{P}q_i : i =1, \ldots n\})$. By density, there exists a sequence $(k_\ell)$ such that {$(\omega_{k_{\ell}+1})$ converges to $\omega$}. For $x \in X$, this implies that
\[ \lim_{\ell \to \infty} \Vert{T^{m_{k_\ell}}(x) - x\Vert} = \lim_{\ell \to \infty} \vert\langle\omega_{k_{\ell}+1},\mathbb{P}x \rangle\vert =   \vert\langle\omega,\mathbb{P}x \rangle\vert.
\]
Hence,  $\mathfrak{L}((m_{k_\ell})_\ell)$ is the kernel of the linear form $\omega^\ast\circ \mathbb{P}$ and $\{q_1,\ldots, q_n\} \in \mathfrak{L}((m_{k_\ell})_\ell)$.
In particular, $\mathfrak{L}((m_{n_\ell})_\ell)$  is not dense and therefore, $\eta(T) = \infty$ by Proposition \ref{prop:eta}.
So it remains to show that $\mathrm{St}(T) = \emptyset$. Suppose the contrary, that is there exists $p \in \mathrm{St}(T)$. Since $T$ is rigid, $Y_{p}=X$, which is a contradiction as $T$ is not hyper-recurrent.
\end{proof}

We would like to point out that the construction in the proof is clearly inspired by the one in \cite{Sophie}. The only difference is that we are considering a Hilbert space instead of a Banach space and that in our case, $E$ has infinite dimension. We would also like to remark that we are not aware of examples of operators with $1< \eta(T) < \infty$.

\section{Exploring hyper-recurrence} \label{sec:exploring}

In this section, we analyse the existence and properties of hyper-recurrent vectors for specific classes of operators. That is, we study power bounded operators in \ref{subsec:pbo}. multiplication operatores in \ref{subsec:mo},  operators with discrete spectrum  \ref{subsec:ods} and composition operators acting on holomorphic functions in
\ref{subsec:co}.

\subsection{Power bounded operators}  \label{subsec:pbo}

Recall from Definition \ref{def:power-bounded} that a linear and continuous operator on a Fréchet space is power-bounded if, for all $x \in X$, the orbit $\mathcal{O}_x:= \{T^n(x): n \in \mathbb{N}\}$ is bounded in $X$. In other words, for any neighborhood $U$ of $0$, there exists a positive number $r$ such that for all $z \in \mathbb{C}$ with $|z|\geq r$, $\mathcal{O}_x \subset zU$. Furthermore, if $T$ is a bounded operator on a Banach space, then $T$ is power bounded if $\sup_n \Vert{T^{n}\Vert}< \infty$.

In the context of power bounded operators on Fréchet spaces, there are the following two important tools. Recall that the Banach-Steinhaus theorem guarantees the equicontinuity of the family $\{T^{n}\}$ whenever $T$ is power bounded \cite[Theorem 2.6]{Rudin:1991}. Moreover, if $d$ is a translation-invariant metric on $X$ and $T$ is power-bounded, then
\begin{eqnarray*}
d_{T}^{*}(x,y):=\sup_{\ell\geq 0}d(T^{\ell}x,T^{\ell}y)
\end{eqnarray*}
defines a translation-invariant metric on $X$ with $d_{T}^{*}(Tx,Ty) \leq d_{T}^{*}(x,y)$ for all $x,y \in X$, which is equivalent to $d$ by the Banach-Steinhaus theorem.

If one adds recurrence in the setting of Banach spaces, it follows from Proposition 3.2 \cite{Cos}) that  $T^{-1}$ is also power bounded and recurrent. This result for Banach spaces motivates us to obtain the same behavior for Fréchet spaces.

\begin{proposition} \label{prop:inverse-of-power-bounded}
Let $T$ be a power-bounded operator on a Fréchet space $X$.
\begin{enumerate}
 \item \label{item:powerbounded-1} If $T$ is recurrent, then $T$ is invertible and $T^{-1}$ is power-bounded and recurrent. Moreover, $T$ is an isometry with respect to $d_{T}^{*}$.
 \item \label{item:powerbounded-2} If $T$ is hyper-recurrent, then $\mathrm{Hr}(T)=\mathrm{Hr}(T^{-1})$ and  $T^{-1}$ is hyper-recurrent.
\end{enumerate}
\end{proposition}
\begin{proof}
We begin with the proof of \ref{item:powerbounded-1}.
As $T$ is recurrent, there exists, for each $x,y\in X$, a strictly increasing sequence $(a_{n})_{n\in \mathbb{N}}$ such that $T^{a_{n}}(x-y)$ converges to $(x-y)$.
Furthermore, as  $d_{T}^{*}(Tx,Ty)\geq d(T^{a_{n}}x,T^{a_{n}}y)$ for each $n\in \mathbb{N}$, we have that  $d_{T}^{*}(Tx,Ty)\geq d(x,y)$. Thus, $d_{T}^{*}(x,y)=d_{T}^{*}(Tx,Ty)$.

Note that a sequence is a Cauchy sequence for $d$ if and only if it is a Cauchy sequence for $d_{T}^{*}$ as the metrics are equivalent. In particular, if $x\in X$ and $(a_{n})_{n\in \mathbb{N}}$ with $\lim_n T^{a_{n}}x =n$, then $(T^{a_{n}}x)$ is a Cauchy sequence with respect to $d_{T}^{*}$. As $T$ is an isometry, $(T^{a_{n}-1}x)$ is a Cauchy sequence as well. Hence, $Ty =x$ for $y:= \lim_n T^{a_{n}-1}x$.
Since $T$ is a continuous bijection, it follows from the open mapping theorem
that $T$ is invertible.

Moreover, $\hbox{Rec}(T^{-1})= \hbox{Rec}(T)$ due to the simple fact that
\begin{equation} \label{eq:corollary-of-isometry}
d_{T}^{*}(T^{{n}}x,x) = d_{T}^{*}(x, T^{{-n}}x)
\end{equation}
for all $x \in X$ and $n \in \mathbb{N}$. As $\hbox{Rec}(T) =X$ by Proposition \ref{esp}, the remaining statement  of \ref{item:powerbounded-1} is proven. Assertion  \ref{item:powerbounded-2} follows from a further application of the identity
in \eqref{eq:corollary-of-isometry}.
\end{proof}

For power-bounded operators in separable Hilbert spaces, we can assume that  $T$ is unitary  (see \cite[Remark 9.4]{Cos}). By the spectral theorem for unitary operators, this implies that $T$ is conjugated to a multiplication operator on a $L^{2}$-space with finite measure. That is, there exists a measure space $(Y,\mathcal{C},\nu)$ where $\nu$ is a non-negative finite Borel measure and a function $u\in L^{\infty}(Y,\nu)$ with $\vert{u\vert}=1$ and a unitary map $\Phi: H\rightarrow L^{2}(Y,\nu)$ such that
   \begin{eqnarray*}
   M_{u} =\Phi\circ T\circ \Phi^{-1}:L^{2}(Y,\nu) \to  L^{2}(Y,\nu),
   f \mapsto  uf.
   \end{eqnarray*}
Observe that, since $M_{u}$ and $T$ are conjugated by an isometry, $\mathrm{Hr}(T)=\Phi^{-1}(\mathrm{Hr}(M_{u}))$. We now show that recurrent, power bounded operators in separable Hilbert spaces are always hyper-recurrent.

\begin{theorem} \label{Hr-H} \label{dim-finita}
Let $H$ be a separable, complex Hilbert space and $T\in \mathcal{B}(H)$ a power-bounded, recurrent operator. Then $T$ is hyper-recurrent and $\hbox{Hr}(T)$ is dense in $H$.
\end{theorem}

\begin{proof} Without loss of generality, assume that $T$ is a multiplication operator on $L^{2}(Y,\nu)$ with $u\in L^{\infty}(Y,\nu)$ and $\vert{u\vert}=1$. For any $f\in L^{2}(Y,\nu)$ and $\epsilon>0$, consider $g$ defined by
\begin{eqnarray*}
g(x) :=
\begin{cases}
f(x)\; &:  \; \vert{f(x)\vert}\geq \frac{\epsilon}{2\nu(Y)^{1/2}},\\
\tfrac{\epsilon}{2\nu(Y)^{1/2}} \;& :\; \vert{f(x)\vert}< \frac{\epsilon}{2\nu(Y)^{1/2}}.
\end{cases}
\end{eqnarray*}
Clearly, $\Vert{f-g\Vert}_{L^{2}(\nu)}<\epsilon$. Furthermore, as $|g|$ is bounded from below, if $ \omega \in \mathfrak{C}$ with $g\in \mathfrak{L}(\omega)$, then  $\int_{Y}\vert{u^{\omega_{n}}-1\vert}^{2}d\nu\xrightarrow[n\longrightarrow\infty]{}0$. Hence, $L^{\infty}(Y,\nu) \subset  \mathfrak{L}(\omega)$ and $g \in \mathrm{Hr}(T)$. As $\epsilon$ can be chosen arbitrarily,  $\mathrm{Hr}(T)$ is dense in $X$.
\end{proof}

The following illustrates which kind of rigidity one may encounter in the presence of hyper-recurrent vectors and a unitary operator on a Hilbert space. In here, we say that $u \in \mathbb{T}$ and the set $\Lambda \subset \mathbb{T}$ are \emph{log-rationally dependent} if there exists a finite subset $\{\lambda_1, \ldots, \lambda_k \} \subset \Lambda$ and $n_1, \ldots n_{k+1} \in \mathbb{Z}$ such that  $u^{n_{k+1}} = \prod_{i=1}^k \lambda_i^{n_i}$ and $n_i \neq 0$ for at least one $i=1, \ldots, k+1$.

\begin{proposition}\label{ortho}
Let $T: H\rightarrow H$ be a recurrent and power-bounded operator on a separable Hilbert space $H$ (i.e. $T$ is recurrent and unitary) such that there exists an invariant closed subspace  $V\subset H$ and an orthonormal basis $\{e_{i}\}_\mathcal{I}$,  for some $\mathcal{I} \subset \mathbb{N}$, of eigenvectors of $T\vert_{V}$ with eigenvalues $\{\lambda_{i}\}_\mathcal{I}$.
With respect to this basis of $V$, the following assertions are equivalent.
\begin{enumerate}
    \item The vector $\sum_{i\in \mathcal{I}}h_{i}e_{i}\in V$ is hyper-recurrent in $(H,T)$ for a sequence $(h_{i})_{i\in \mathcal{I}}$ contained in $\mathbb{C}$ such that $h_{i}\neq 0$ for all $i\in \mathcal{I}$.
    \item $\{\sum_{i}{g_{i}e_{i}}\in V: g_{i}\neq 0 \;\forall i\in \mathcal{I}\}\subset \mathrm{Hr}(T)$
    \item $(H, T)$ is conjugate by an unitary operator to $\left(\left(\bigoplus_{\alpha\in F}\mathcal{H}_{\alpha})\right)_{2},S\right)$, where $F \subset \mathbb{T}$ is a countable set such that $\{\lambda_{i}:i \in \mathcal{I}\} \subset F$ and each $\alpha \in F$ is log-rationally dependent from $\{\lambda_{i}\}$. Moreover, each $\mathcal{H}_{\alpha}$ is a Hilbert space and $S|_{\mathcal{H}_{\alpha}}(v) = \alpha v$.
\end{enumerate}
\end{proposition}

\begin{proof} The implication from (2) to (1) is immediate. For the backward direction, assume that $y\in V$  is a hyper-recurrent vector with $y=\sum_{i\in \mathcal{I}}g_{i}e_{i}$ with $g_{i}\neq 0$ for all $i \in \mathcal{I}$. Hence, if $y\in \mathfrak{L}(\omega)$ for some sequence $\omega\in \mathfrak{C}$, then
\begin{eqnarray*}
    \Vert{T^{\omega_{n}}y-y\Vert}^{2}=\sum_{i\in \mathcal{I}}\vert{g_{i}\vert}^{2}\vert{\lambda_{i}^{\omega_{n}}-1\vert}^{2}\longrightarrow 0.
\end{eqnarray*}
In particular, this is equivalent to $\lim_n \lambda_{i}^{\omega_{n}} =1$ for each $i\in \mathcal{I}$. As this equivalence only requires that the $g_{i}\neq 0$, it follows that (1) implies (2).

We now show that (1) implies (3). By the spectral theorem for unitary operators, $(H,T)$ is conjugated by a unitary operator $\Phi$ to $(L^{2}(Y, \nu), M_{u})$.
For $\beta\in \mathbb{T}$, set
\begin{eqnarray*}
L(\beta) : =\{\omega=(\omega_{n})_{n\in \mathbb{N}}: \omega_{n}\in \mathbb{N}, \omega_{n} \uparrow \infty\; \hbox{and}\; \beta^{\omega_{n}}\longrightarrow 1\}
\end{eqnarray*}
For $x = \sum h_i e_i \in \mathrm{Hr}(T)$ given by (1), it then follows as above that
\[
\Phi(x) \in \mathfrak{L}_{M_{u}}(\omega) \iff  \omega \in \bigcap_{i \in \mathcal{I} } L(\lambda_{i}).
\]
\begin{claim}
 $u(y)$ and $\langle{(\lambda_{i})_{i\in \mathcal{I}}\rangle}$ are log-rationally dependent for $\nu$-almost every $y\in Y$.
 \end{claim}
 If $\omega\in \mathfrak{C}$ with $ \Phi(x)\in \mathfrak{L}_{M_{u}}(\omega)$, then $\mathfrak{L}_{M_{u}}(\omega)=L^{2}(Y,\nu)$. In particular, for the function $1\in L^{2}(Y,\nu)$, we have that
 \begin{eqnarray*}
 \Vert{M_{u}^{\omega_{n}}1-1\Vert}^{2}=\int_{Y}\vert{(u(y))^{\omega_{n}}-1\vert}^{2}d\nu(y) \xrightarrow{n\rightarrow \infty}  0.
 \end{eqnarray*}
Hence, there exists a subsequence $\theta:=(\theta_{n})_{n}$ of $\omega$ such that $(u(y))^{\theta_{n}}$ converges to $1$ as $n\rightarrow \infty$ for $\nu$-almost every $y\in Y$.
Or, in other words, there exists a  mensurable set $Y^{\prime}\subset Y$ such that $\nu(Y^{\prime})= \nu(Y)$ and
$\theta \in L(u(y))$ for all  $y\in Y^{\prime}$. Hence,
\begin{equation} \label{eq:here is theta}
\theta \in \bigcap_{y\in Y^{\prime}}   L(u(y)) \cap \bigcap_{i \in \mathcal{I} } L(\lambda_{i}).
\end{equation}

Now assume, for $y\in Y^{\prime}$, that $u(y)$ is log-rationally independent from  $\{ \lambda_{i}: {i\in \mathcal{I}} \}$. It then follows as in the proof of Proposition \ref{RnH} that there exists  $\beta\in \mathbb{T}\setminus\{1\}$ such that there is a strictly increasing sequence of positive integers $(\psi_{\ell})_{\ell\in \mathbb{N}}$ with
\begin{eqnarray*}
    (u(y))^{\psi_{\ell}}\xrightarrow{\ell\rightarrow\infty} \beta, \; \;\lambda_{i}^{\psi_{\ell}}\xrightarrow{\ell\rightarrow\infty}1, \forall i\in \mathcal{I},
\end{eqnarray*}
which is a contradiction to \eqref{eq:here is theta}. This proves the claim.

\medskip

\noindent This now gives rise to the following construction. Set
\begin{align*}
 F:= \left\{\alpha \in \mathbb{T} : \alpha \hbox{ and } \{ \lambda_{i}\}  \; \hbox{ are $\mathbb{Q}$-dependent},\;  \nu(\{y\in Y: u(y)=\alpha\})>0\right\}.
\end{align*}
As the set of finite subsets of a countable set and $\mathbb{Z}$ are countable, it follows that $F$ is at most countable. It now follows from our claim  that $\{y \in Y : u(y) \in F\}$ is a set of full measure. In particular, we have that, for each $f \in L^{2}(Y,\nu)$,
\[
M_{u}f = \sum_{\alpha \in F} \alpha \left(f \chi_{\{y: u(y)=\alpha\}}  \right).
\]
In particular, one obtains the desired conjugation with respect to $\mathcal{H}_{\alpha} := L^2(\{y: u(y)=\alpha\},\nu)$
Moreover, as  $M_{u}(\Phi(e_{i}))=\lambda_{i}\Phi(e_{i})$, there exists  for each $i \in \mathcal{I}$ an $\alpha \in F$ with $\alpha=\lambda_{i}$. This proves that (1) imples (3).

The remaining direction, that is (3) implies (1), is not very difficult to verify.
\end{proof}

We now show, that in the context of an isometry on a complex Banach, hyper-recurrence implies that
$\mathrm{Hr}(T)$ is not very small.

\begin{proposition} Let $X$ be an infinite-dimensional complex Banach space and $T\in \mathcal{L}(X)$ be a hyper-recurrent and surjective isometry. Then, the set $\overline{\mathrm{Hr}(T)}$ contains an infinite-dimensional $T$-invariant subspace.
\end{proposition}

\begin{proof} Note that $\mathrm{Rec}(T)=X$ and $\mathfrak{L}(\omega)$ always is closed as isometries are power-bounded.

We begin with showing  that, if $x$ is a hyper-recurrent vector for $T$, then $\hbox{span}\{\mathcal{O}_{T}(x)\}\subset \overline{\mathrm{Hr}(T)}$ and that $\alpha x+\beta Tx$ is also hyper-recurrent for each $\alpha, \beta\in \mathbb{C}\textstyle{\setminus\{0\}}$ with $\vert{\alpha\vert}\neq \vert{\beta\vert}$.
Without loss of generality, assume that $\vert{\alpha\vert}<\vert{\beta\vert}$. For $\omega\in \mathfrak{C}$ with $\alpha x+\beta Tx$ in $\mathfrak{L}(\omega)$, we use the triangular inequality and that $T$ is an isometry to obtain that
\begin{eqnarray*}
 \Vert{T^{\omega_{n}}(\alpha x+\beta Tx)-(\alpha x+\beta Tx)\Vert} \geq (\vert{\beta\vert}-\vert{\alpha\vert}) \Vert{T^{\omega_{n}}x-x\Vert} \xrightarrow[n\rightarrow \infty]{}0.
\end{eqnarray*}
Hence $x \in \mathfrak{L}(\omega)$ and, as $x \in \mathrm{Hr}(T)$,  $\mathfrak{L}(\omega)=X$. In particular, $\alpha x+\beta Tx \in \mathrm{Hr}(T)$.

We now use this result to show that $P(T)(x)$ belongs to $\overline{\mathrm{Hr}(T)}$ for any polynomial $P$. So assume that $P$ is of the form $ P(t)=\xi \prod_{i=1}^{m}(t-\lambda_{i})$ for some complex numbers $\xi, \lambda_{1}, \ldots, \lambda_{m}$. If $|\lambda_{i}| \neq 1$ for all $i$, then the above implies that
$P(T)(x) \in {\mathrm{Hr}(T)}$. Furthermore, in order to analyse the general case, it suffices to approximate $P$ as follows.
Assume that  $Q_{k}(t)=\xi \prod_{i=1}^{m}(t-\lambda_{i,k})$ is a polynomial $Q_{k}$ of degree $m$ with roots $\{\lambda_{i,k}\}_{i=1}^{m}$ such that the sequence $\lambda_{i,k}$ converges to $\lambda_{i}$ as $k\rightarrow \infty$ and $\vert{\lambda_{i,k}\vert}\neq 1$. Then
$Q_{k}(T)(x) \to P(T)(x)$ and $Q_{k}(T)(x) \in \mathrm{Hr}(T)$, which then implies that $P(T)(x) \in \overline{\mathrm{Hr}(T)}$.

The following two scenarios arise for $E_x := \mathrm{span}(\{T^{n}x: n\geq 0\})$. If $E_x$ has infinite dimension, then $\overline{\mathrm{Hr}(T)}$ contains an infinite-dimensional $T$-invariant subspace. On the other hand, if the dimension of $E_x$ is finite, then $T|_{E_{x}}$ is conjugated to a diagonal matrix with unimodular entries \cite[Theorem 4.1]{Cos}. Hence, there exists a finite basis of eigenvectors $\{e_i\}$ with eigenvalues $\lambda_i$. Moreover, by construction, $x = \sum h_ie_i$ with $h_i \neq 0$ for all $i$ and $x \in \mathfrak{L}(\omega)$ if and only if $\lim_n \lambda_i^{\omega_n} = 1$ for all $i$.

As $E_x$ has  finite dimension, there exists a closed subspace $L \subset X$ such that $X = E_{x} \oplus L$. For $q \in L$ and $p = \sum u_i e_i$ with $u_i \neq 0$ for all $i$, we then have, for $\omega_n \uparrow \infty$, that
\begin{align*}
\lim_{n \to \infty} T^{\omega_n}(p+q) = p+q & \iff \lim_{n \to \infty} T^{\omega_n}(p)  = p \; \mathrm{ and }  \; \lim_{n \to \infty} T^{\omega_n}(q)  = q\\
& \iff \lim_{n \to \infty} T^{\omega_n}(x)  = x  \; \mathrm{ and }  \; \lim_{n \to \infty} T^{\omega_n}(q)  = q  \\ & \iff \lim_{n \to \infty} T^{\omega_n}(x+q) = x+q.
\end{align*}
Hence, $p+q \in \mathfrak{L}(\omega)$ if and only if $x+q \in \mathfrak{L}(\omega)$. It now follows from hyper-recurrence of $x$ that  $\mathfrak{L}(\omega) = X$ whenever $x \in \mathfrak{L}(\omega)$. It follows from this that  $ p+q \in \mathfrak{L}(\omega) = X$ for any $q \in L$ and $p$ as above.
This concludes the proof.
\end{proof}

\subsection{Multiplication Operators} \label{subsec:mo}

Note that the spectral theorem for normal operators on separable Hilbert space implies that the operator is conjugated to a multiplication operator. In order to generalize this setting, we now consider operators of the form
\begin{eqnarray*}
M_{\phi}: X\to X, \; f \mapsto \phi f,
\end{eqnarray*}
where $X$ is a Banach space and $\phi f$ is well-defined pointwise multiplication. For example, in case of a normal operator, one may consider $X = L^2(X,\nu)$ and $\varphi \in L^\infty(X,\nu)$. A further classical example are multiplication operators on unital Banach spaces, i.e. Banach Algebras with a multiplicative identity element.
%
Recall that, in a  unital Banach algebra $X$, the set of invertible elements, denoted by $G(X)$, is open (cf. \cite[Theorem 10.12]{Rudin:1991}).

\begin{theorem}\label{multi}
Let $X$ be a unital Banach algebra, $\theta\in X$ and let $M_{\theta}$ refer to the operator given by  multiplication  either from the left or from the right. Then  $M_{\theta}$ is recurrent if and only if $M_{\theta}$ is hyper-recurrent. In both cases, $G(X)\subset \mathrm{Hr}(M_{\theta})$.
\end{theorem}

\begin{proof} Assume $M_{\theta}$ is a left-multiplication operator on $X$ and is recurrent. Then, there exists a recurrent vector $x$ in $G(X)$ since $G(X)$ is open in $X$. By the open mapping theorem, the right multiplication $R_{x}: X\rightarrow X,  y \to yx$ is an isomorphism. Hence, there exists a constant $c(x)>0$ such that $\Vert{yx\Vert}\geq c(x)\Vert{y\Vert}$ for all $y\in X$.

Let us show that $x$ is a hyper-recurrent vector. Choose $\omega\in \mathfrak{C}$ such that $x\in \mathfrak{L}(\omega)$. Then
\begin{eqnarray*}
\Vert{M_{\theta}^{\omega_{n}}x-x\Vert}=\Vert{\theta^{\omega_{n}}x-x\Vert}\geq c(x)\Vert{\theta^{\omega_{n}}-e\Vert}\xrightarrow{n\rightarrow \infty}  0,
\end{eqnarray*}
which implies that the sequence $\theta^{\omega_{n}}$ converges to $e$. Therefore, $\mathfrak{L}(\omega)=X$, as for each $y\in X$,
\begin{eqnarray*}
\Vert{M_{\theta}^{\omega_{n}}y-y\Vert}=\Vert{\theta^{\omega_{n}}y-y\Vert}\leq \Vert{\theta^{\omega_{n}}-e\Vert} \Vert{y\Vert}\xrightarrow{n\rightarrow \infty} 0.
\end{eqnarray*}
Consequently, every element in $X$ is a recurrent vector and $x \in \mathrm{Hr}(M_{\theta})$. This implies that $G(X)\subset \mathrm{Hr}(M_{\theta})$.
\end{proof}

\begin{example}\label{mul-conexo}
Consider a compact and connected metric space $(K,d)$. For any $\phi\in C(K)$, if the left-multiplication operator $M_{\phi}$ is recurrent, then there exists $\beta\in \mathbb{T}$ such that $\phi(p)=\beta$ for every $p\in K$ (see \cite[Theorem 7.1]{Cos}). Now, let us consider a continuous non-zero function $g: K\rightarrow \mathbb{C}$. It is clear that there exists some $q\in K$ with $g(q)\neq 0$. If $g\in \mathfrak{L}(\omega)$ for some sequence $\omega$, then
\begin{eqnarray*}
    \Vert{M_{\phi}^{\omega_{n}}g-g\Vert}_{\infty}=\sup_{p\in K}{\vert{\beta^{\omega_{n}}g(p)-g(p)\vert}}\geq \vert{\beta^{\omega_{n}}-1\vert}\vert{g(q)\vert}\xrightarrow{n \rightarrow \infty}0
\end{eqnarray*}
Consequently, we can conclude that $\beta^{\omega_{n}}$ converges to $1$ as $n\rightarrow \infty$, which implies that $\mathfrak{L}(\omega)=C(K)$.
In particular, $\mathrm{Rec}(M_{\phi})=C(K)$ and $\mathrm{Hr}(M_{\phi})=C(K)\setminus\{0\}$.
\end{example}

\begin{example}
If $K$ is a compact Hausdorff space, then $C(K)$ has a dense invertible group if and only if $dim(K)$, the covering dimension of $K$, is not more than 1 (see  \cite{Dawson}). Under these conditions, for $\phi\in C(K)$, if the multiplication operator $M_{\phi}$ on $C(K)$ is recurrent, then $\mathrm{Hr}(M_{\phi})$ is dense in $C(K)$ by Theorem \ref{multi}.
\end{example}

Now we turn our attention to the multiplication operator on a Banach space of holomorphic functions defined on a non-empty open subset of $\mathbb{C}$. In the work of Costakis et al. \cite[Subsection 7.4]{Cos}, they analyzed the recurrence of the multiplication operator on Banach spaces of holomorphic functions on the unit disk $\mathbb{D}$. More precisely, they showed in the context of the Hardy spaces $H^{p}(\mathbb{D})$ and Bergman spaces $A^{p}(\mathbb{D})$ for $1\leq p<\infty$, the Bloch space $\mathcal{B}$, and the Dirichlet space $\mathcal{D}$ that the multiplication operator is recurrent if and only if $\phi$ is a constant in $\mathbb{T}$. The proof in there makes implicitly use of the following notion.

\begin{definition}
    Let $X$ be a nontrivial Banach space of holomorphic functions in a given open set $\Omega\subset \mathbb{C}$. We say that each point-evaluation functional is bounded if, for each $z\in \Omega$, there exists a constant $C(z)>0$ such that
        $\vert{f(z)\vert}\leq C(z)\Vert{f\Vert},\;\;\forall f\in X$.
\end{definition}
\begin{example} \label{Example:Hardy-space}
For $1 \leq p<\infty$, the Hardy space $H^{p}(\mathbb{D})$ consists of all holomorphic functions $f$ on $\mathbb{D}$ such that
\begin{eqnarray*}
\Vert{f\Vert}_{H^{p}(\mathbb{D})}:=\sup_{0\leq r<1}\left(\frac{1}{2\pi}\int_{0}^{2\pi} \vert{f(re^{i\theta})\vert}^{p}d\theta\right)^{1/p}<\infty.
\end{eqnarray*}
With respect to this norm, $H^{p}(\mathbb{D})$ becomes a Banach space. Moreover, for each $f\in H^{p}(\mathbb{D})$, we have for each $ z\in \mathbb{D}$ that
\begin{eqnarray*}
\vert{f(z)\vert}\leq\tfrac{\Vert{f\Vert}_{H^{p}(\mathbb{D})}}{(1-\vert{z\vert}^{2})^{1/p}}.
\end{eqnarray*}
For the other spaces, namely the Bergman, Bloch and Dirichlet spaces, the growth estimate can be found in \cite{Ale}. For some $\phi\in X$, consider the multiplication operator $M_{\phi}:X \rightarrow X$. If $M_{\phi}$ is recurrent then, there exists $\beta\in \mathbb{T}$ such that $\phi(z)=\beta$ for each $z\in \mathbb{D}$ by \cite[Theorem 7.4 and 7.5]{Cos}. It is not very hard to show that $M_{\phi}$ is hyper-recurrent.
\end{example}

The ideas presented in the above example motivated us to extend certain concepts from its proof to Banach spaces of holomorphic functions defined on a general  non-empty open set. The following result both extends and generalizes Theorems 7.5 and 7.6 in \cite{Cos} in two ways.
First of all, we consider abstract spaces of holomorphic functions defined on open sets with bounded point-evaluation functional instead of the Hardy, Bergman, Bloch and Dirichlet spaces.  Moreover, we add hyper-recurrence to the list of equivalences.

\begin{theorem}\label{Banach-holo}
Let $\Omega\subset \mathbb{C}$ be an open set and $\{A_{i}\}_{i\in I}$ its connected components.
Furthermore, assume that  $(X,\Vert \cdot {\Vert})$ is a non-trivial Banach space of holomorphic functions on $\Omega$ such that  $1\in X$, each point-evaluation functional is bounded and, for every $g\in X$,  we have $\{g \chi_{A_i} \}_{i}\subset X$ and $\sum_{i \in I}\Vert{ g \chi_{A_i} \Vert}<\infty$.
Then, for a holomorphic function $\phi$ on $\Omega$ such that the multiplication operator $M_{\phi}: X\rightarrow X$ is well-defined, the following are equivalent.
\begin{enumerate}
    \item  $M_{\phi}$ is recurrent,
    \item $M_{\phi}$ is rigid,
    \item  $\phi|_{A_i}$ is constant for all $i\in I$ and $\vert{\phi(z)\vert}=1$ for all $z\in \Omega$,
    \item $M_{\phi}$ is hyper-recurrent,
    \item $\hbox{Hr}(M_{\phi})$ is dense in $X$.
\end{enumerate}
\end{theorem}

\begin{proof} Throughout the proof, we will write $f_i := g \chi_{A_i}$ and assume that $I \subset \mathbb{N}$ in order to have an order at hand.

We now show that (1) implies (3). So assume that $M_{\phi}$ is recurrent. Then, for each  $i\in I$, there exists an $f\in \mathrm{Rec}(M_{\phi})$ such that $f_{i}$ is not identically zero. Consequently, we have $(M_{\phi}^{a_{n}}f)_{i}\xrightarrow[n\rightarrow\infty]{} f_{i}$ for some $\omega=(a_{n})_{n}\in \mathfrak{C}$. By boundedness of the point-evaluation functional, there exists for each $z\in A_{i}$ with $f_{i}(z)\neq 0$ a positive constant $c(z)$ such that
\begin{eqnarray*}
\vert{f_{i}(z)}\vert\vert{\phi_{i}(z)^{a_{n}}-1\vert}\leq c(z)\Vert{(M_{\phi}^{a_{n}}f)_{i}-f_{i}\Vert}\xrightarrow{n\rightarrow\infty} 0.
\end{eqnarray*}
Hence,  $\phi_{i}(z)^{a_{n}}\xrightarrow{n\rightarrow \infty} 1$ and $\vert{\phi_{i}(z)\vert}=1$ for each $z\in A_{i}$ with $f_{i}(z)\neq 0$. As $Z(f_{i})=\{z\in A_{i}: f_{i}(z)=0\}$ is countable, $\vert{\phi(z)\vert}=1$ for each $z\in A_{i}$. Therefore, by the open mapping theorem for holomorphic functions, there exists $b_{i}\in \mathbb{T}$ such that $\phi(z)=b_{i}$ for each $z\in A_{i}$. Consequently, (1) leads to (3)

\begin{claim}
If $\omega=(a_{n})_{n}$ satisfies $b_{i}^{a_{n}}\xrightarrow{n\rightarrow \infty} 1$ for each $i\in I$, then $\mathfrak{L}(\omega)=X$.
\end{claim}
\noindent For any $\epsilon>0$ and $g\in X$ choose $i_{0}\in I$ such that $\sum_{i>i_{0}}\Vert{g_{i}\Vert}<\epsilon/4$. Moreover, there is $n_{0}$ such that $\vert{b_{i}^{a_{n}}-1}\vert\leq \frac{\epsilon}{2i_{0}(1+\max_{i\leq i_{0}} \Vert{g_{i}\Vert})}$ for $1\leq i\leq i_{0}$ and $n> n_0$. Hence, for $n>n_{0}$,
\begin{eqnarray*}
\Vert{M_{\phi}^{a_{n}}g-g\Vert}\leq \sum_{i\in I}\vert{b_{i}^{a_{n}}-1\vert}\Vert{g_{i}\Vert}\leq \sum_{i\leq i_{0}} \vert{b_{i}^{\ell_{n}}-1\vert}\Vert{g_{i}\Vert}+2\sum_{i>i_{0}} \Vert{g_{i}\Vert}<\epsilon.
\end{eqnarray*}
This proves the claim.

Furthermore, as the rotation on any finite dimensional torus is recurrent, a simple diagonal argument shows that such a sequence $\omega$ exists. This proves that (3) implies (2). Now, we will show that (3) implies (4) and (5).

\begin{claim}
Any $h\in X$ with $h_{i}\neq 0$ for each $i\in I$ is hyper-recurrent.
\end{claim}
\noindent If  $h\in \mathfrak{L}(\omega)$ for some $\omega=(a_{n}) \in \mathfrak{C}$, then for each $i\in I$ we have $(M_{\phi}^{a_{n}}h)_{i}\xrightarrow{n\rightarrow \infty} h_{i}$. For $z\in A_{i}$ with $h_{i}(z)\neq 0$, there is a positive constant $c(z)>0$ such that
\begin{eqnarray*}
\vert{h_{i}(z)\vert}\vert{b_{i}^{a_{n}}-1\vert}=\vert{h_{i}(z)}\vert\vert{\phi_{i}(z)^{a_{n}}-1\vert}\leq c(z)\Vert{(M_{\phi}^{a_{n}}h)_{i}-h_{i}\Vert}\xrightarrow[n\rightarrow\infty]{}0.
\end{eqnarray*}
Hence $b_{i}^{a_{n}}\xrightarrow[n\rightarrow \infty]{}1$ for each $i\in I$. By the above claim, we have $\mathfrak{L}(\omega)=X$.

The remaining implications are trivial.
\end{proof}

\begin{example}
For each $k\in \mathbb{N}$, let $D_{k}:=\{z\in \mathbb{C}: \vert{z\vert}<k\}$. Define $A(D_{k})$, the disk algebra on the disk $D_{k}$, as
\begin{eqnarray*}
    A(D_{k}):=\{f:\overline{D_{k}}\rightarrow \mathbb{C}: f\; \text{is continuous on}\; \overline{D_{k}}\; \text{and holomorphic on}\; D_{k}\}
\end{eqnarray*}
equipped with the norm $\Vert{\cdot\Vert}_{k}$ given by $\Vert{f\Vert}_{k}:=\sup_{z\in D_{k}}\vert{f(z)\vert}$ for $f\in A(D_{k})$. Equipped with pointwise multiplication, $(A(D_{k}), \Vert{\cdot\Vert}_{k})$ forms a unital Banach algebra. For $\phi\in A(D_{k})$, let $M_{\phi}$ be the multiplication operator on $A(D_{k})$. We can use two approaches to show that if $M_{\phi}$ is recurrent, then it is hyper-recurrent. The first case is by observing that $A(D_{k})$ is a unital Banach algebra, and then $M_{\phi}$ is hyper-recurrent by Theorem \ref{multi}. The second case is obtained by observing that $(A(D_{k}), M_{\phi})$ satisfies the conditions of the Theorem \ref{Banach-holo}. Therefore, $\phi\equiv\beta$ for some $\beta\in \mathbb{T}$. With a similar argument as in Example \ref{mul-conexo}, we can conclude that $\mathrm{Hr}(M_{\phi})=A(D_{k})\setminus\{0\}$.
\end{example}

To conclude this subsection, we will examine multiplication operators on $L^p$-spaces. Costakis et al. characterized the recurrence of the multiplication operator on $L^p$-spaces of finite measure for  $1 \leq p < \infty$ through a certain behavior of the function $\phi\in L^{\infty}$ (see Theorem 7.6 in \cite{Cos}).
In here, we now extend the result to spaces equiped with  a  $\sigma$-finite measure and again add hyper-recurrence to the list of equivalences.

\begin{theorem} \label{theo:mult}
Let $(Y, \nu)$ be a measure space where $\nu$ is a non-negative $\sigma$-finite measure. For $\phi\in L^{\infty}(Y, \nu)$ consider the multiplication operator $M_{\phi}: L^{p}(Y,\nu)\rightarrow L^{p}(Y, \nu)$ for some $1 \leq p<\infty$. The following are equivalent.
\begin{enumerate}
    \item \label{no:theo:mult:1} $M_{\phi}$ is recurrent,
    \item \label{no:theo:mult:2} $M_{\phi}$ is rigid,
    \item \label{no:theo:mult:3} there exists a strictly increasing sequence $(a_{n})_{n\in \mathbb{N}}$ such that $\lim_{n \to \infty} \phi^{a_{n}} =1$ almost surely,
    \item \label{no:theo:mult:4} $M_{\phi}$ is hyper-recurrent,
    \item \label{no:theo:mult:5} $\hbox{Hr}(M_{\phi})$ is dense in $L^{p}(Y, \nu)$.
\end{enumerate}
\end{theorem}

\begin{proof} If $\mu(X)=\infty$, then there exists a partition of $X$ into mensurable sets $\{A_{j}\}_{j\in \mathbb{N}}$ with $0<\nu(A_{j})<\infty$ for each $j\in \mathbb{N}$.
We now show that, if $M_{\phi}$ is recurrent, then $\vert{\phi\vert}=1$ almost surely. In order to do so, for $m, j\in \mathbb{N}$, set  $A_{j,m}:=\{y\in A_{j}:\vert\phi(y)\vert>1+\frac{1}{m}\}$. If $\nu(A_{j,m})>0$ for some a positive integer $m$, then there is $f\in \hbox{Rec}(M_{\phi}))$ such that $\int_{A_{j,m}}\vert{f\vert}^{p}d\nu=h>0$. We can see that for each $\ell\in \mathbb{N}$:
\begin{eqnarray*}
\Vert{M^{\ell}_{\phi}f-f\Vert}^{p}_{L^{p}(Y,\nu)}\geq \int_{E_{j,m}}\vert{(\phi(y))^{\ell}-1\vert}^{p}\vert{f(y)\vert}^{p}d\nu(y)\geq\left(\frac{\ell}{m}\right)^{p}h,
\end{eqnarray*}
which is a contradiction since $f$ is recurrent. Thus $\nu( \bigcup_{j,m} A_{j,m})=0$. A similar argument shows that $\nu(\{y\in Y: \vert{\phi(y)\vert}<1\})=0$.

In particular, the multiplication operator $M_{\phi}$ is a isometry. Therefore, $\hbox{Rec}(M_{\phi})=L^{p}(Y, \nu)$. Now assume that $g \in
L^{p}(Y, \nu)$ satisfies $c_i := \inf_{y \in A_i} |g(y)| > 0$ for all $i \in \mathbb{N}$. Then, for $\omega = (a_n) \in \mathfrak{C}$ with $g\in \mathfrak{L}(\omega)$, we have
\begin{eqnarray*}
\Vert{M_{\phi}^{a_{n}}g-g\Vert}^{p}_{L^{p}(Y,\nu)}=\sum_{j=1}^{\infty} \int_{E_{j}}\vert{\phi(y)^{a_{n}}-1\vert}^{p}\vert{g(y)\vert}^{p} d\nu \geq \sum_{j=1}^{\infty}c_{j}^{p} \int_{E_{j}}\vert{\phi(y)^{a_{n}}-1\vert}^{p}d\nu\xrightarrow{n\rightarrow \infty} 0,
\end{eqnarray*}
which implies that $\lim_n \phi(y)^{a_{n}} =1$ for $\nu$-almost everywhere $y$. This proves that \ref{no:theo:mult:1} implies \ref{no:theo:mult:3}. The implication from \ref{no:theo:mult:3} to \ref{no:theo:mult:2} now is a consequence of Lebesgue's dominated convergence theorem. The remaining statements follow from the observation that
$\inf_{y \in A_i} |g(y)| > 0$ for all $i \in \mathbb{N}$ implies that $g$ is hyper-recurrent.
\end{proof}

\subsection{Operators with discrete spectrum}\label{subsec:ods}

We now study the hyper-recurrence of operators with the property that the vector space generated by the set $\mathcal{E}(T):=\{x\in X: Tx=\lambda x\;\hbox{for some}\; \lambda\in \mathbb{T}\}$ is dense in $X$. Moreover, recall that $T$ is said to have \emph{discrete spectrum} if $\overline{\mathrm{span}(\mathcal{E}(T))}=X$.

\begin{proposition} Let $T$ be a linear and continuous operator with discrete spectrum on a complex, separable Fréchet space $X$.
Then $T$ is quasi-rigid. In particular, if $T$ is power-bounded then $T$ is rigid.
\end{proposition}

\begin{proof}
For any finite set $\{\lambda_{j}\}_{j\in F}\subset \mathbb{T}$, there exists a strictly increasing sequence of positive integers $(a_{n})_{n\in \mathbb{N}}$ such that for each $j\in F$, the sequence $\lambda_{j}^{a_{n}}$ converges to $1$ as $n\rightarrow \infty$. It is easy to see that for every $m\in \mathbb{N}$, the set of all recurrent vectors of $T_m$ contains the dense set $(\mathrm{span}(\mathcal{E}(T)))^m$. By Theorem \ref{equiv}, there exists a sequence $\omega\in\mathfrak{C}$ such that $\overline{\mathfrak{L}(\omega)}=X$. The remaining assertion follows from Proposition \ref{esp}.
\end{proof}

When studying aspects of recurrence, one can approach it by considering stronger notions of recurrence, like uniform, reiterativ, frequent or upper frequent recurrence, as well as recurrence along certain Furstenberg families (see, e.g., \cite{muro, Boni, Fhyp, Grivaux, Sophie}). In here, we only will work with the following notion which is related to  $\mathcal{E}(T)$ by a result of Grivaux and López-Martínez in \cite{Grivaux}.
\begin{definition}
 Let $X$ be a Banach space and let $T: X \rightarrow X$ be a continuous operator. We say that $x\in X$ is uniformly recurrent for $T$ if for any non-empty open set $U$   of $x$, the set $N(x, U)$ is a syndetic set. We will denote by $\mathrm{URec}(T)$ the set of such points and refer to $T$ as uniformly recurrent if $\mathrm{URec}(T)$ is dense in $X$.
\end{definition}

Grivaux and López-Martínez obtained that $\overline{\hbox{span}(\mathcal{E}(T))}=\overline{\mathrm{URec}(T)}$ for power-bounded operators on complex reflexive Banach spaces (\cite[Th. 1.9]{Grivaux}). Moreover, it follows from the work of  Bonilla and Grosse-Erdmann (\cite[Corollary 3.2]{Boni}) that  either $\mathrm{URec}(T)$ is of first category or $\mathrm{URec}(T) =X$, where $T$ is a uniformly recurrent operator on a Banach space $X$.
However, it is not known whether uniformly recurrence on a complex Banach spaces always implies that $\hbox{span}(\mathcal{E}(T))$ is dense in $X$ (\cite[Question 1.8]{Grivaux}).

We now present two results in which the space generated by $\mathcal{E}(T)$ is dense. The first result ensures that if the operator $T$ is invertible, then $T$ is hyper-recurrent. The second result states that if the set of periodic points is dense, then the operator is hyper-recurrent, noting that it is not required for the operator to be invertible.

\begin{theorem}\label{disc-esp}
Let $T$ be a complex separable Fréchet space and $T\in \mathcal{L}(X)$ be an invertible operator with discrete spectrum (i.e. $
\overline{\mathrm{span}(\mathcal{E}(T))}=X$). Then  $T$ is hyper-recurrent and $\overline{\mathrm{Hr}(T)} =X$.
\end{theorem}

\begin{proof}
As $X$ is a complex Fréchet space, the metric on $X$ can be defined using a countable family of seminorms $\{p_{k}\}_{k\in \mathbb{N}}$ associated with $X$. For $x, y\in X$, we have
\[
d(x,y)=\sum_{k=1}^{\infty} \frac{1}{2^{k}}\frac{p_{k}(x-y)}{1+p_{k}(x-y)}
\]
As a consequence, for $\lambda\in \mathbb{T}$ and $x\in X$, $d((\lambda-1)x,0)\leq d(2x,0)$.

{Now assume that $\alpha$ is an eigenvalue of the operator $T$. For any vector $q$ with $Tq=\alpha q$, set $A(q) := \mathbb{C}q$. $A(q)$ possesses a topological complement, denoted by $M(q)$. As $T$ is invertible, it follows that $M_{q}$ is $T$-invariant. Let $\mathbb{P}_{q}$ be a continuous projection operator onto $A(q)$. It is worth noting that for each eigenvalue $\beta\neq \alpha$ of the operator $T$, we have $\{y\in X: Ty=\beta y\}\subset \hbox{Ker}(\mathbb{P}_{q})$. On the other hand, by the separability of $X$, there exists a countable set of eigenvalues $G=\{\lambda_{i}\}_{i\in \mathbb{N}}\subset \mathbb{T}$ of the operator $T$ such that
\begin{align*}
\overline{\hbox{span}\{x\in X: Tx=\lambda x \;\hbox{for some}\; \lambda\in G\}}=X.
\end{align*}
\begin{claim}
If $\omega:=(\omega_{n})_{n}$ is a sequence such that $\lim_n \lambda^{\omega_{n}}=1$ for each $\lambda\in G$, then $\overline{\mathfrak{L}(\omega)}=X$.
\end{claim}
\noindent This follows immediately by noting that the span of vectors $x\in X$ such that $Tx=\lambda x$ for some $\lambda\in G$ is contained in $\mathfrak{L}(\omega)$.}

\medskip
{Now we will show that $\mathrm{Hr}(T)$ is dense in $X$. So assume that $y=\sum_{j=1}^{m}{q_{j}}$ for some  $q_j \in \mathcal{E}(T)$ and $\beta_{j} \in \mathbb{T}$ with  $Tq_{j} = \beta_{j} q_{j}$, for $1 \leq j \leq m$.
Moreover, assume that $0 < \epsilon  < \min |q_j|/m$ and choose  $\xi_i \neq 0$ with $T\xi_{i}=\lambda_{i}\xi_{i}$ for each $i\in \mathbb{N}$ such that $\sum_{i=1}^{\infty}d(\xi_{i},0)<\epsilon$.
\begin{claim}
    $z := y + \sum_{i=1}^{\infty} \xi_{i} $ is hyper-recurrent.
\end{claim}
For the countable set $G\cup \{\beta_{j}\}_{j}^{m}\subset \mathbb{T}$, there is a strictly increasing sequence of positive integers $\theta:=(\theta_{n})_{n}$ such that $\lambda_{i}^{\theta_{n}}$ and $\beta_{j}^{\theta_{n}}$ converge to 1 as $n\rightarrow \infty$ for each $i,j$. We will show that $z$ is recurrent. For any fixed $\delta>0$, there exists $i_{0}$ such that $\sum_{\ell>i_{0}}d(2\xi_{\ell},0)<\delta/2$. In addition, there exists $n_{0}$ such that for each $n>n_{0}$ we have $\sum_{j=1}^{m}d(T^{\theta_{n}}q_{j},q_{j})+\sum_{\ell=1}^{i_{0}}d(T^{\theta_{n}}\xi_{\ell},\xi_{\ell})<\delta/2$. Therefore, for  $n>n_{0}$,
\begin{align*}
	d(T^{\theta_{n}}z,z) &\leq \sum_{j=1}^{m}d(T^{\theta_{n}}q_{j},q_{j})+\sum_{\ell=1}^{i_{0}}d(T^{\theta_{n}}\xi_{\ell},\xi_{\ell})+\sum_{\ell>i_{0}}d(T^{\theta_{n}}\xi_{\ell},\xi_{\ell})\\
	{} &< \delta/2+\sum_{\ell>i_{0}}d((\lambda_{\ell}^{\theta_{n}}-1)\xi_{\ell},0)
	\stackrel{(\ast)}{<} \delta/2+\sum_{\ell>i_{0}}d(2\xi_{\ell},0)
	< \delta,
\end{align*}
where we have used that the metric was constructed from seminorms in $(\ast)$.
We will now demonstrate that $z$ is hyper-recurrent. If $z\in \mathfrak{L}(\omega)$ for a sequence $\omega:=(\omega_{n})_{n}$, it then follows that  $T^{\omega_{n}}(y+\sum_{\ell=1}^{\infty} \xi_{\ell}) \xrightarrow[n\rightarrow \infty]{} y+\sum_{\ell=1}^{\infty} \xi_{\ell}$. Hence, For each $i\in \mathbb{N}$,
\begin{align*}
	\mathbb{P}_{\xi_{i}}\left( \sum_{j=1}^{m}\beta_{j}^{\omega_{n}}q_{j}+\sum_{\ell=1}^{\infty} \lambda_{\ell}^{\omega_{n}} \xi_{\ell}\right) &\xrightarrow[n\rightarrow \infty]{} \mathbb{P}_{\xi_{i}}\left(\sum_{j=1}^{m}q_{j}+\sum_{\ell=1}^{\infty} \xi_{\ell}\right),\\
	\lambda_{i}^{\omega_{n}}\xi_{i} &\xrightarrow[n\rightarrow \infty]{}  \xi_{i},\\
	\lambda_{i}^{\omega_{n}} &\xrightarrow[n\rightarrow \infty]{} 1.
\end{align*}
Thus, $\mathfrak{L}(\omega)$ is dense in $X$.}
\end{proof}

\begin{theorem}\label{per-dense}
Let $X$ be a Fréchet Space and $T\in\mathcal{L}(X)$. If $\overline{\hbox{Per}(T)}=X$
then $T$ is a hyper-recurrent operator and $\hbox{Hr}(T)$ is dense in $X$.
\end{theorem}

\begin{proof}
Let $\mathcal{A}(T)$ refer to the subset of $\mathbb{N}$  such that $n$ is the (prime) period of some periodic vector. That is, for each $n \in \mathcal{A}(T)$, there exists $x \in X$ with $T^n x = x$ and $T^k x \neq x$ for $0 < k < n$. In this situation, we write $\mathfrak{p}(x):= n$.

\begin{claim} If $\mathcal{A}(T)$ is finite, then $T^{\max \mathcal{A}(T)} = \mathrm{Id}$ and $\hbox{Hr}(T)$ is dense in $X$.
\end{claim}
\noindent If $\mathcal{A}(T)$ is finite, then combining the pigeonhole principle with the density of periodic points, it follows that for each $x \in X$, there exists $k \in \mathcal{A}(T)$ and a sequence $(y_n)$ of periodic points of period $k$ converging to $x$. Hence, by continuity, $\mathfrak{p}(x) \leq k \in \mathcal{A}(T)$. In fact, $\mathfrak{p}(x)$ divides $k$.

We now show that $A := \{ x:\mathfrak{p}(x) = \max \mathcal{A}(T)\}$ is dense. However, if $A$ would not be dense, then there would exist $x \in A$ and $y_n \to x$ with $\mathfrak{p}(x) > \mathfrak{p}(y_n)$ for all $n$, a contradiction. Hence, it follows from the density of $A$ and the above argument that each element in $\mathcal{A}(T)$ divides $\max \mathcal{A}(T)$. The claim now follows from the observation that $A = \hbox{Hr}(T)$.

\begin{claim} If $\mathcal{A}(T)$ is infinite, then  $\hbox{Hr}(T)$ is dense in $X$.
\end{claim}
\noindent We now fix an invariant metric $d$ on $X$. For $p$ a periodic vector, set
\[
a(p) := \min_{0\leq i<j<\mathfrak{p}(p)}d(T^{i}p,T^{j}p), \quad
b(p) := \max_{0\leq i<j<\mathfrak{p}(p)}d(T^{i}p,T^{j}p).
\]
Now assume that $\mathcal{A}(T)=\{ m_{i} : i \in \mathbb{N} \}$ for a strictly increasing sequence $(m_i : i \in \mathbb{N})$. Hence, using the fact that scalar multiplication in $X$ is continuous, there exists  a sequence $(p_{\ell} : \ell \in \mathbb{N} )$ in $\hbox{Per}(T)$ that such
 \begin{enumerate}
     \item $\mathfrak{p}(p_{\ell}) = m_{\ell}$ for all $\ell \in \mathbb{N}$,
     \item $4b(p_{\ell+1})<a(p_{\ell})$ for all $\ell \in \mathbb{N}$,
     \item $q:=\sum_{\ell \in \mathbb{N}}p_{\ell}\in X$.
 \end{enumerate}
Observe that $q \in \mathfrak{L}((n!))$ as, by translation invariance of the metric,
\begin{eqnarray*}
d(T^{n!}q, q) & = & d(\sum_{m_{\ell}>n}T^{n!}p_{\ell}, \sum_{m_{\ell}>n}p_{\ell})) \leq  \sum_{m_{\ell}>n}d(T^{n!}p_{\ell},p_{\ell})\\
{} &\leq & \sum_{m_{\ell}>n}b(p_{\ell})
< a(p_{1})\left(\sum_{m_{\ell}>n}\frac{1}{4^{\ell}}\right)\xrightarrow{n\rightarrow \infty}0.
\end{eqnarray*}
We can see that if $q\in \mathfrak{L}(\omega)$ for a sequence $\omega:=(\omega_{n})_{n}$, then for each $m_{\ell}$ there is $n_{0}$ such that $m_{\ell}\vert \omega_{n}$ for $n\geq n_{0}$. The proof is by induction. Let's show only the case $\ell=1$,
\begin{eqnarray*}
d(T^{\omega_{n}}q,q) &\geq& d(T^{\omega_{n}}p_{1}, p_{1})-\sum_{\ell>1}d(T^{w_{n}}p_{\ell},p_{\ell})\\
{} &\geq & d(T^{\omega_{n}}p_{1},p_{1})-\frac{a(p_{1})}{12}
\end{eqnarray*}
As a consequence, $d(T^{\omega_{n}}p_{1},p_{1})<a(p_{1})/2$
with $n\geq n_{0}$ for some positive integer $n_{0}$. Therefore, $m_{1}\vert \omega_{n}$ with $n\geq n_{0}$. This is, if $q\in \mathfrak{L}(\omega)$ for a sequence $\omega:=(\omega_{n})_{n}$, then for each $m_{\ell}$ there exists an integer $n_0$ such that $m_{\ell}$ divides $\omega_n$ for all $n\geq n_0$. Recall that, $\hbox{Per}(T)=\left(\bigcup_{\ell\geq 1} \hbox{Ker}(T^{m_{\ell}}-Id)\right)\bigcup \{\hbox{fixed points}\}$. It is clear that $\hbox{Ker}(T^{m_{\ell}}-Id)\subset \mathfrak{L}(\omega)$ for every $\ell\geq 1$. Thus, $\mathfrak{L}(\omega)$ is a dense set of $X$ and therefore, $q \in \hbox{Hr}(T)$. Let us now prove the density of the set $\hbox{Hr}(T)$ in $X$. Take $x\in \hbox{Ker}(T^{m_{1}}-Id)$, and consider $y=x+t(\sum_{\ell>1}p_{\ell})\in \hbox{Hr}(T)$ for $t\ll 1$. It's evident that $\hbox{Ker}(T^{m_{1}}-Id)\subset \overline{\hbox{Hr}(T)}$, and this argument can be extended to the case where $\ell>1$.
\end{proof}

\begin{example}
For a sequence $\lambda=(\lambda_{k})_{k\in\mathbb{N}}\in\ell^{\infty}(\mathbb{N})$, we define the diagonal operator $T_{\lambda}$ given by
\begin{eqnarray*}
T_{\lambda}: (X,\Vert{.\Vert}_{\infty}) &\longrightarrow & (X, \Vert{.\Vert}_{\infty}),\\
(x_{k})_{k} &\longmapsto & (\lambda_{k}x_{k})_{k},
\end{eqnarray*}
where $X=c_{0}(\mathbb{N})$ or $X=\ell^{p}(\mathbb{N})$ with $1\leq p\leq \infty$. If $T_{\lambda}$ is recurrent, then $T_{\lambda}$ is an invertible operator by Theorems 5.3 and 5.4 in \cite{Cos}. Therefore, $T_{\lambda}$ is hyper-recurrent and $\mathrm{Hr}(T_{\lambda})$ is dense in $X$ by Theorem \ref{disc-esp}.
\end{example}

\begin{example}
	Let $X$ be a separable Banach space and $T\in \mathcal{L}(X)$. We say that $(X, T)$ is Devaney chaotic if $T$ is topologically transitive and the periodic points are dense in $X$. In this setting, we have that $\mathrm{Hr}(T)\setminus \hbox{HC}(T)$, with $\hbox{HC}(T)$ referring to the set of hypercyclic vectors, is dense in $X$ since we can choose hyper-recurrent vectors with bounded orbit as in the proof of Theorem \ref{per-dense}.
\end{example}

\subsection{Composition operators on spaces of holomorphic functions} \label{subsec:co}

Let $X$ be a Banach or Fréchet space consisting of holomorphic functions from $\Omega\subset \mathbb{C}$ to $\mathbb{C}$, and $\phi:\Omega\rightarrow \Omega$. The composition operator $C_{\phi}:X\rightarrow X$ associated to $\varphi$ is then defined as $g \mapsto g \circ \phi$, provided that the operation is well-defined.

We begin with reminding the reader that the space of holomorphic functions $\mathrm{H}(\Omega)$ on an
open subset $\Omega \subset \mathbb{C}$ is a Fréchet space with respect to the topology of uniform convergence on compact sets.
In here, $\Omega$ will be either $\mathbb{C}$, $\mathbb{D}$, or $\mathbb{C}^{*} := \mathbb{C}\setminus \{0\}$. Furthermore, we will consider $X \subset \mathrm{H}(\Omega)$, where $X$ is either $\mathrm{H}(\Omega)$, the Hardy space  $\mathrm{H}^{2}(\mathbb{D})$ (cf. Example \ref{Example:Hardy-space} in Subsection \ref{subsec:mo}), the weighted Hardy space of entire functions $E^{p}(\beta)$ (see \eqref{def:entire-weighted-functions}  below) or the Hardy space of Dirichlet series $\mathcal{H}$ (see \eqref{eq:Hardy-space-of-dirichlet} below).

\begin{theorem}\label{composition}
Consider $X=\mathrm{H}(\mathbb{C}), \mathrm{H}(\mathbb{C}^{*}), \mathrm{H}(\mathbb{D}),\mathrm{H}^{2}(\mathbb{D}),  E^{p}(\beta)$, or $\mathcal{H}$. If the composition operator $C_{\phi}$ on $X$ is recurrent, then $C_{\phi}$ is a hyper-recurrent operator and $\mathrm{Hr}(C_{\phi})$ is dense in $X$.
\end{theorem}

For the proof, we consider three distinct groups: Group 1 consists of $\mathrm{H}(\mathbb{C})$, $\mathrm{H}(\mathbb{C}^{*})$, $\mathrm{H}(\mathbb{D})$ and $\mathrm{H}^{2}(\mathbb{D})$, group 2 comprises the spaces $E^{p}(\beta)$ and group 3 only contains $\mathcal{H}$.

\subsubsection*{Main reductions steps for group 1}

For $\phi \in H(\mathbb{C})$, the composition operator $C_{\phi}: \mathrm{H}(\mathbb{C})\rightarrow \mathrm{H}(\mathbb{C})$ is recurrent if and only if $\phi(z)=az+b$ with $a,b\in \mathbb{C}$ and $\vert{a\vert}=1$ (Th. 6.4 in \cite{Cos}).
This leads to three cases: If $a=1$ and $b\neq 0$, then $C_{\phi}$ is hypercyclic (see \cite[Example 1.4]{livro}), if $a=1$ and $b=0$, then $C_{\phi} = \mathrm{Id}$, whereas $a\neq 1$ requires further analysis.

For an automorphism $\phi$ of $\mathbb{C}^{*}$, when the composition operator $C_{\phi}$ on $\mathrm{H}(\mathbb{C}^{*})$ exhibits recurrence, two possibilities arise: either $\phi(z)=a/z$ for some $a\in \mathbb{C}^{*}$ or $\phi(z)=az$ for some $a\in \mathbb{T}$ \cite[Theorem 6.6]{Cos}. It is evident that analyzing the second case is sufficient.

For the spaces $\mathrm{H}(\mathbb{D})$ and $\mathrm{H}^{2}(\mathbb{D})$, we consider specific holomorphic functions $\phi: \mathbb{D}\rightarrow \mathbb{D}$ corresponding to well-known linear fractional maps. In other words, we can express $\phi$ in the form $\phi(z) = ({az+b})/({cz+d})$, where $a,b,c,d \in \mathbb{C}$ with $ad - bc \neq 0$. A classification of these maps can be found in \cite{Sha}, among which the elliptic automorphisms stand out. Recall that if $\phi$ is an elliptic automorphism, then there exists a linear fractional map $S$ and $\lambda\in \mathbb{T}$ such that $\phi=S^{-1}\psi S$, where $\psi(z)=\lambda z, z\in \mathbb{D}$. Furthermore, if the composition operator $C_{\phi}$ on $\mathrm{H}(\mathbb{D})$ or $\mathrm{H}^{2}(\mathbb{D})$ is recurrent, it falls into one of two scenarios: either $C_{\phi}$ is hypercyclic or $\phi$ is an elliptic automorphism, as discussed in \cite[Theorems 6.9 and 6.12]{Cos}. Hence, it is sufficient to focus on the scenario where the map $\phi$ is an elliptic automorphism. We also refer the reader to the books \cite{Sha, Cowen} for a reference to the theory of composition operators on Hardy and other spaces.

\begin{table}[h!]
    \centering
\renewcommand{\arraystretch}{1.12}
     \begin{tabular}{ l @{$\quad$} l }
          Form of Map &   Spaces \\
        \hline
        $\phi(z)=az+b, a\in \mathbb{T}\setminus \{1\}$ and $b\neq 0$ & $\mathrm{H}(\mathbb{C})$  \\
        $\phi=az, a\in \mathbb{T}$ &  $\mathrm{H}(\mathbb{C}), \mathrm{H}(\mathbb{D}), \mathrm{H}(\mathbb{C}^{*}), \mathrm{H}^{2}(\mathbb{D})$
    \end{tabular}
\end{table}

\begin{proof}[Proof of Theorem \ref{composition} for Group 1] For $\mathrm{H}^{2}(\mathbb{D})$, it suffices to employ Theorem \ref{Hr-H}, whereas the remaining cases are consequences of the Baire category theorem.
\begin{enumerate}
    \item Assume that $\phi(z)=az$ with $a\in \mathbb{T}$ acts on $\mathrm{H}^{2}(\mathbb{D})$.
    Then, as the norm of $f \in \mathrm{H}^{2}(\mathbb{D})$ is given by (see Example \ref{Example:Hardy-space})
    \begin{eqnarray*}
  \Vert{f\Vert}_{H^{2}(\mathbb{D})}=\sup_{0\leq r<1}\left( \frac{1}{2\pi}\int_{0}^{2\pi}\vert{f(re^{i\theta})\vert}^{2}d\theta\right)^{1/2}<\infty,
    \end{eqnarray*}
    it immediately follows that $C_{\phi}$ is an isometry. It is now evident that $C_{\phi}$ is a surjective isometry on a separable Hilbert space.
    Consequently, by Theorem \ref{Hr-H}, $C_{\phi}$ is hyper-recurrent and $\mathrm{Hr}(C_{\phi})$ is dense in $\mathrm{H}^{2}(\mathbb{D})$.
    \item Assume that $\phi(z)=az$ with $a\in \mathbb{T}$ acts on $X=\mathrm{H}(\Omega)$ where $\Omega=\mathbb{C},\mathbb{C}^{*}$ or $\mathbb{D}$ and consider the set
    \begin{eqnarray*}
   \mathcal{A}=\bigcup_{\theta}\{ f\in X: f(\theta z)=f(z)\, \forall z\in \Omega\}
    \end{eqnarray*}
    where $\theta\neq 1$ runs through all the roots of unity. It's worth noting that the set $\mathcal{A}$ is a countable union of closed subspaces that are never dense in $X$. Therefore, by the Baire Category Theorem, the set $\mathrm{Rec}(T)\setminus \mathcal{A}$ is dense in $X$. If $a^{\omega_{n}}$ converges to $1$ for some strictly increasing sequence of positive integers $\omega:=(\omega_{n})_{n}$, then $C_{\phi}^{\omega_{n}}f$ converges to $f$ as $n$ goes to infinity in the topology of uniform convergence on compact sets, this is, $\mathfrak{L}(\omega)=X$. It's not difficult to verify that $\mathrm{Rec}(T)\setminus \mathcal{A}$ is contained in $\mathrm{Hr}(T)$.
    \item  Assume that $\phi(z)=az+b$ with $a\in \mathbb{T}\setminus \{1\}$ and $b\neq 0$ acting on $\mathrm{H}(\mathbb{C})$ and consider the set
    \begin{eqnarray*}
        \mathcal{F}=\bigcup_{\theta}\left\{ f\in \mathrm{H}(\mathbb{C}): f\left(\theta z + \tfrac{\theta-1}{a-1}b\right)=f(z), \forall z\in \mathbb{C}\right\}
    \end{eqnarray*}
    where $\theta\neq 1$ runs through all the roots of unity. By similar arguments to the previous item, we have that the set $\mathrm{Rec}(T)\setminus \mathcal{F}$ is dense in $\mathrm{H}(\mathbb{C})$ and  contained in $\mathrm{Hr}(C_{\phi})$.
\end{enumerate}
\end{proof}

\subsubsection*{Main reduction steps for group 2}

We now recall the definition of the \emph{weighted Hardy space of entire functions}.
Let $\beta=(\beta_{n})_{n\geq 0}$ be a sequence of positive real numbers such that $((n+1)\beta_{n}/\beta_{n+1})_{n\geq 0}$ is bounded. This sequence induces a Banach space of entire functions through
\begin{eqnarray}
\label{def:entire-weighted-functions}
E^{p}(\beta)=\left\{f=\sum_{n=0}^{\infty}{a_{n}z^{n}}: \Vert{f\Vert}=\left( \sum_{n=0}^{\infty}\vert{a_{n}\vert}^{p}\beta_{n}^p\right)^{1/p}<\infty\right\},
\end{eqnarray}
where $1<p<\infty$. The polynomials form a dense subset of $E^{p}(\beta)$. For $p=2$, this setting corresponds to the weighted Hardy space, initially introduced in \cite{Chan}. The hypercyclic behavior of translation operators within this space has been studied in \cite{Chan}, and the boundedness and compactness of composition operators on entire functions of several variables have been explored in \cite{Doan}. The boundedness of such operators relies on the behavior of the sequence of weights, as discussed in \cite{TanBoun}.

To study the composition operator over the space $E^{p}(\beta)$, we first aim to determine the form of the entire function $\phi$ when the composition operator $C_{\phi}$ is continuous on the Banach space $E^{p}(\beta)$. In analogy to the results in \cite{Doan, Cars, TanBoun}, there are  high expectations that these operators have a simple form.

\begin{proposition}
Assume that $\phi$ is an entire function on $\mathbb{C}$ such that the composition operator  $C_{\phi}$ acting on $E^{p}(\beta)$ is continuous.
\begin{enumerate}
 \item \label{item-1-prop-group2} There exist $a,b\in \mathbb{C}$ with $\vert{a\vert}\leq 1$ such that $\phi(z)=az+b$.
 \item \label{item-2-prop-group2} If $C_{\phi}$ is recurrent, then $\vert a \vert =1$.
\end{enumerate}
\end{proposition}

\begin{proof}
For each $z\in \mathbb{C}$, consider the linear functional $A_{z}\in (E^{p}(\beta))^{*}$ given by $A_{z}(f)=f(z)$. It follows from Hölder's inequality that
\begin{eqnarray*}
\Vert{A_{z}\Vert} \leq \left( \sum_{n=0}^{\infty}\frac{\vert{z\vert}^{qn}}{\beta^{q}_{n}}\right)^{1/q}
\end{eqnarray*}
where $q$ is the conjugate exponent of $p$. Moreover, for $\Phi_z(t) :=  \sum_{n} \beta_n^{-q} |z|^{qn} z^{-n} t^n$, we have $\| \Phi_z\| < \infty$ as $\beta_n$ grows faster than $n!$, and $\Vert{A_{z}\Vert} = |A_{z}(\Phi_z)| / \| \Phi_z\| = \Phi_1(|z|^q)^{1/q}$.
As $C^{*}_{\phi}(A_{z})=A_{\phi(z)}$ for each $z\in \mathbb{C}$, it follows that
\begin{eqnarray*}
\Vert{C_{\phi}\Vert}= \Vert{C^{*}_{\phi}\Vert}=\sup_{\Lambda\neq 0}\frac{\Vert{C_{\phi}^{*}(\Lambda)\Vert}}{\Vert{\Lambda\Vert}}\geq \sup_{z\in \mathbb{C}\setminus \{0\}}\frac{\Vert{C^{*}_{\phi}A_{z}\Vert}}{\Vert{A_{z}\Vert}}
 = \sup_{z\in \mathbb{C}\setminus \{0\}}  \left( \frac{ \Phi_1(|\phi(z)|) }{ \Phi_1(|z|)} \right)^{1/q}
\end{eqnarray*}
Hence,  ${\Phi(\vert{\phi(z)\vert}^{q})}/{\Phi(\vert{z\vert}^{q})}$ is uniformly bounded. It then follows from Lemma 4.5 in   \cite{Doan} that
\[  \limsup_{|z| \to \infty} |\phi(z)/z| \leq 1. \]
As $\phi$ is an entire function on $\mathbb{C}$, $\phi(z) = az + b$, for some $a,b\in \mathbb{C}$ with $\vert{a\vert}\leq 1$. This proves the assertion in \ref{item-1-prop-group2}. In order to prove \ref{item-2-prop-group2}, assume that $\vert a \vert < 1$. Then
\[ \phi^n(z) = a^n z + a^{n-1}b + \cdots + b  = a^n z + \frac{1  - a^n}{1-a}b
\xrightarrow{n\to\infty} \frac{b}{1-a}. \]
Hence, if $\vert a \vert < 1$, then  the only $C_{\phi}$-recurrent functions are the constant functions.
\end{proof}

We now consider the following four cases. If $a=1$ and $ b\neq 0$, then $C_{\phi}$ is hypercyclic (see  \cite[Exercise 8.1.2]{Karl}). Moreover, if $a=1$ and $b=0$, then $\phi$ is the identity. In the third case, that is $a \neq 1$ is root of unity and $b \in \mathbb{C}$, it follows from
\begin{equation} \label{eq:iterates-of-phi} \phi^n(z) - z = a^n z + \tfrac{1  - a^n}{1-a}b -z = (a^n -1)(z+ b/(a-1)) \end{equation}
that $C_\varphi$ is periodic (i.e. there exists $n> 1$ such that $C_\varphi^n$  is the identity.
In particular, in these three cases, $\mathrm{Hr}(C_{\phi})$ is dense in $E^{p}(\beta)$. In the remaining case, that is $a\neq 1$, it follows from
\eqref{eq:iterates-of-phi} and the norm on $E^{p}(\beta)$ that any polynomial is recurrent.
In order to show that  $\mathrm{Hr}(C_{\phi})$ is dense in $E^{p}(\beta)$ also in this case, we need the following  lemma.

\begin{lemma}\label{Micy}
For $a,b\in \mathbb{C}$ with $a\in \mathbb{T}\setminus\{1\}$ and $b\neq 0$, there exists a dense and uncountable set $M=M(a,b)\subset \mathbb{C}$ such that for any $k\geq 1$ and any pairwise different points $c_{1},\ldots, c_{k}$ in $M$, the following holds:
\begin{eqnarray*}
\sum_{j=1}^{k}c_{j}b^{j}\left(\frac{\theta-1}{a-1}\right)^{j}\neq 0
\end{eqnarray*}
for all   roots of unity $ \theta$ with  $\theta \neq 1$.
\end{lemma}

\begin{proof}
Consider the sets $\mathcal{A}_{n}$ for each $n\geq 1$ defined as:
\begin{eqnarray*}
\mathcal{A}_{n}=\bigcup_{\theta}\left\{(c_{1},\ldots,c_{n})\in \mathbb{C}^{n}: \sum_{j=1}^{n}c_{j}b^{j}\left(\frac{\theta-1}{a-1}\right)^{j}=0\right\}
\end{eqnarray*}
where $\theta\neq 1$ runs through all  roots of unity. For every $n\in \mathbb{N}$, the set $\mathcal{R}_{n}=\mathbb{C}^{n}\setminus \mathcal{A}_{n}$ is a residual set in $\mathbb{C}^{n}$. By the Mycielski Theorem, there exists a dense uncountable set $M\subset \mathbb{C}$ such that $(c_{1},\ldots, c_{k})\in \mathcal{R}_{k}$ for any $k\geq 1$ and any pairwise different $k$ points $c_{1},\ldots, c_{k}$ in $M$.
\end{proof}

\begin{proof}[Proof of Theorem \ref{composition} for Group 2]
It is enough to show that the operator $C_{\phi}$ is hyper-recurrent for $\phi(z)=az+b$ such that $a\in \mathbb{T}\setminus\{1\}$ and $b\in \mathbb{C}$.  It's worth noting that the set of polynomials is contained in $\mathrm{Rec}(C_{\phi})$. We distinguish the two cases  $b=0$ and $b\neq 0$.

\begin{description}
 \item $\mathbf{b \neq  0.}$  In order to show the density of the hyper-recurrent vectors, we will build a dense collection of polynomials which are hyper-recurrent vectors, whose coefficients are obtained using the Lemma \ref{Micy} . That is, for $a\neq 1$ and $b\neq 0$, using the set $M$ of Proposition \ref{Micy}, we let $F$ refer to the set of all non-constant polynomials of the form $P(z)=\sum_{j=0}^{\ell}c_{j}z^{j}$ where $c_{0}\in \mathbb{C}$ and $\{c_{j}\}_{j=1}^{\ell}\subset M$ are pairwise distinct.

\begin{claim}
$F$ is dense in $E^{p}(\beta)$ and $F\subset \mathrm{Hr}(C_{\phi})$.
\end{claim}

It is evident that the set $F$ is dense in $X$. Furthermore, note that convergence of $g_{k}$ to $g$ in $E^{p}(\beta)$ implies that   $g_{k}(z)$ converges to $g(z)$ for each $z\in \mathbb{C}$. Now assume that $g\in F$ belongs to $\mathfrak{L}(\omega)$ for some $\omega:=(\omega_{n})_{n}$. , we will show that the sequence $(a^{\omega_{n}})_{n}$ converges to 1.
\begin{eqnarray*}
g\left(a^{\omega_{n}}z+\frac{a^{\omega_{n}}-1}{a-1}b\right)\xrightarrow[n\rightarrow \infty]{}g(z)
\end{eqnarray*}
for each $z\in \mathbb{C}$. Suppose $a^{\omega_{n}}$ does not converge to 1, then there is $\beta\in \mathbb{T}\setminus\{1\}$ a limit point of $(a^{\omega_{n}})_{n}$ such that $g(\beta z+\frac{\beta-1}{a-1}b)=g(z)$ for each $z\in \mathbb{C}$. We observe that $\beta$ is a root of unity since $g$ is a non-constant polynomial.

Now, we can write $g(z)=c_{0}+\sum_{i=j}^{\ell}c_{j}z^{j}$ for some $c_{0}\in \mathbb{C}$ and $\{c_{j}\}_{j=1}^{\ell}\subset M$. Thus, for each $z\in \mathbb{C}$ we have that
\begin{eqnarray*}
c_{0}+\sum_{j=1}^{\ell}c_{j}\left(\beta z+\frac{\beta-1}{a-1}b\right)^{j}&=&c_{0}+\sum_{i=1}^{\ell}c_{i}z^{j}
\end{eqnarray*}
in particular, for $z=0$ in the above equation,
\begin{eqnarray*}
\sum_{i=1}^{\ell}c_{j}\left(\frac{\beta-1}{a-1}\right)^{j}b^{j}&=& 0
\end{eqnarray*}
This leads to a contradiction. If $a^{\theta_{n}}$ converges to 1 for some $\theta := (\theta_{n})_{n}$, then $\mathfrak{L}(\theta)$ contains all polynomials.  This concludes the claim.

\item $\mathbf{b =  0.}$  It's worth noting that in this case, the composition operator $C_{\phi}$ is an isometry, which implies that $\mathrm{Rec}(T)=E^{p}(\beta)$. Let us consider the dense set $\mathcal{G}$ defined by
\begin{eqnarray*}
\mathcal{G}:=\left\{f\in E^{p}(\beta): f(z)=\sum_{n=0}^{\infty}c_{n}z^{n}\; \text{with}\; c_{1}\neq 0\right\}
\end{eqnarray*}
We now show that  $\mathcal{G}$ is contained in $\mathrm{Hr}(C_{\phi})$. If $f\in \mathcal{G} \cap \mathfrak{L}(\omega)$ for some $\omega\in \mathfrak{C}$, then
\begin{eqnarray*}
\Vert{C^{\omega_{k}}_{\phi}f-f\Vert}^{p}=\sum_{n=0}^{\infty} \vert{c_{n}\vert}^{p}\vert{a^{n\omega_{k}}-1\vert}^{p}\beta_{n}^{p} \xrightarrow[k\rightarrow \infty]{} 0.
\end{eqnarray*}
Hence $a^{\omega_{n}}$ converges to $1$ and therefore, $\mathfrak{L}(\omega)=E^{p}(\beta)$.
\end{description}
 \end{proof}

\subsubsection*{Main Steps for Group 3}

The Hardy space of Dirichlet series $\mathcal{H}$ is defined as
\begin{eqnarray} \label{eq:Hardy-space-of-dirichlet}
\mathcal{H}=\left\{f(s)=\sum_{n=1}^{\infty}\frac{a_{n}}{n^{s}}\; \text{with}\; \|f\| = \left(\sum_{n=1}^{\infty}\vert{a_{n}\vert}^{2} \right)^{1/2} <\infty\right\}
\end{eqnarray}
and was introduced by Hedenmalm, Lindqvist, and Seip in \cite{Lind}. As it is well-known, Dirichlet series are well defined on half-spaces of the form $\mathbb{C}_\theta := \{s\in \mathbb{C}:\mathfrak{R}(s)>\theta\}$, and if $\theta \in \mathbb{R}$ is minimal with this property, then
$\theta$ is referred to as the abscissa of convergence of the series  (for further details on the theory of Dirichlet series, we refer to \cite{Queff, Defant}). In the case of $f \in \mathcal{H}$, the Cauchy-Schwarz inequality implies that $f$ converges in the
the half-plane $\mathbb{C}_{\sfrac12}$ and in fact is holomorphic in there (\cite{Lind}).

Gordon and Hedenmalm characterized the composition operators on $\mathcal{H}$ in \cite{Heden} through the following result:
\begin{theorem}[Gordon and Hedenmalm, \cite{Heden}]
An analytic function $\phi: \mathbb{C}_{\sfrac12}\rightarrow \mathbb{C}_{\sfrac12}$ defines a continuous composition operator $C_{\phi}:\mathcal{H}\rightarrow \mathcal{H}$ if and only if
\begin{enumerate}
    \item The function is of the form
    $\phi(s)=c_{0}s+\varphi(s)$,
    where $c_{0}\in \mathbb{N}\cup \{0\}$  and $\varphi(s)=\sum_{1}^{\infty}c_{n}n^{-s}$ admits a representation by a Dirichlet series that is convergent in some half-plane $\mathbb{C}_\theta$.
    \item $\phi$ has an analytic extension to $\mathbb{C}_{0} = \{s :\mathfrak{R}(s)>0\}$, also denoted by $\phi$, such that
      \[ \phi(\mathbb{C}_{0})\subset \begin{cases}
                                            \mathbb{C}_{0} & \mathrm{if}\; c_{0}\geq 1, \\
                                             \mathbb{C}_{\sfrac12}  & \mathrm{if}\; c_{0} = 0.
                                     \end{cases}
      \]
\end{enumerate}
\end{theorem}
It follows from \cite[Remark p. 329]{Heden} that the composition operator is contractive if and only if $c_{0} \in \{1,2,3, \ldots\}$. In the ensuing result, we provide a characterization of recurrence for composition operators on the Hardy space of Dirichlet series.

\begin{theorem}\label{dirichlet} 
Let $\phi$ be an analytic self-map of $\mathbb{C}_{\sfrac12}$ such that the composition operator $C_{\phi}:\mathcal{H}\rightarrow \mathcal{H}$ is continuous. Then $C_{\phi}$ is recurrent  if and only if $\phi(s)=s+it$ for some $t\in \mathbb{R}$. In particular, if $C_{\phi}$ is recurrent, then the operator is hyper-recurrent and $\hbox{Hr}(C_{\phi})$ is dense  $\mathcal{H}$.
\end{theorem}

\begin{proof}
We {first} show that $C_{\phi}$ is not recurrent if $\phi(s)\neq s+it$.  Indeed,
if $\phi(s)\neq s+it$ for any $t\in \mathbb{R}$, then there exists $\eta>0$ such that $\phi(\mathbb{C}_{\sfrac12})\subset \mathbb{C}_{\sfrac{1}{2} + \eta}$ by Lemma 11 in  \cite{Bayart}. Let $r:=r(\eta)$ be given by
\begin{eqnarray*}
r=2+\frac{2^{\frac{1}{2} +\eta}(1+2(\zeta(1+\eta))^{1/2})}{2^{\frac{\eta}{2}}-1},
\end{eqnarray*}
where $\zeta$ represents the zeta function. Then, $f(s)=r2^{-s}\notin \overline{\hbox{Rec}(C_{\phi})}$ due to the following arguments. Suppose the contrary. Then there exists a recurrent vector $g(s)=\sum_{1}^{\infty}a_{n}n^{-s}$ with $\Vert{g-f\Vert}<1$. Thus, we have
\begin{eqnarray*}
r-1 < \vert{a_{2}\vert},\quad
\sum_{n\neq 2}\vert{a_{n}\vert}^{2} < 1.
\end{eqnarray*}
Furthermore, $g(\phi^{\omega_{k}}(\frac{1+\eta}{2}))$ converges to $g(\frac{1+\eta}{2})$ as $k\rightarrow \infty$ for some $\omega=(\omega_{k})\in \mathfrak{C}$. However, for each $\ell\in \mathbb{N}$,
\begin{eqnarray*}
\left\vert{\sum_{n=1}^{\infty}\frac{a_{n}}{n^{\frac{1+\eta}{2}}}-\sum_{n=1}^{\infty}\frac{a_{n}}{n^{\phi^{\ell}(\frac{1+\eta}{2})}}}\right\vert &\geq& \vert{a_{2}\vert}\left\vert{\frac{1}{2^{\frac{1+\eta}{2}}}-\frac{1}{2^{\phi^{\ell}(\frac{1+\eta}{2})}}}\right\vert-\sum_{n\neq 2}\left(\frac{\vert{a_{n}}\vert}{n^{\frac{1+\eta}{2}}}+\frac{\vert{a_{n}}\vert}{n^{\mathfrak{R}(\phi^{\ell}(\frac{1+\eta}{2}))}}\right)\\
{} &\geq& (r-1)\left(\frac{1}{2^{\frac{1+\eta}{2}}}-\frac{1}{2^{\frac{1}{2}+\eta}}\right)-2 (\zeta(1+\eta)^{1/2}) > 1,\\
\end{eqnarray*}
which is a contradiction.

Hence, we assume from now on that  $\phi(s)=s+it$. In particular, as it easily can be verified, $C_\phi$ then is a surjective isometry.
To proceed, we select a strictly increasing sequence $\omega=(\omega_{k})_{k\in \mathbb{N}}$ such that
\begin{eqnarray*}
n^{it\omega_{k}}\xrightarrow[n\rightarrow \infty]{} 1
\end{eqnarray*}
for each $n\geq 2$ and aim to show that $\mathfrak{L}(\omega)=\mathcal{H}$. So assume that $f(s)=\sum_{n=1}^{\infty} {a_{n}}{n^{-s}}$ is in $\mathcal{H}$.
Then
\begin{eqnarray*}
\Vert{C_{\phi}^{\omega_{k}}f-f\Vert}^2=\sum_{n=1}^{\infty}\vert{a_{n}\vert}^{2}\vert{n^{-it\omega_{k}}-1\vert}^{2}\xrightarrow[k\rightarrow \infty]{} 0.
\end{eqnarray*}
Hence, $\mathrm{Rec}(C_\phi) = \mathcal{H}$. It now follows from Theorem \ref{Hr-H} that $C_{\phi}$ is hyper-recurrent and $\hbox{Hr}(C_{\phi})$ is dense.
\end{proof}

The following statement is a consequence of the above proof as an isometry never is  hypercyclic. For an alternative proof, we refer to \cite[Proposition 4.1]{Bayartcomp}.
\begin{corollary}
No composition operator on $\mathcal{H}$ is hypercyclic.
\end{corollary}

We conclude the paper with the following open problems.
\begin{question}
Let $X$ be an infinite-dimensional separable Banach space. Is it true that for any quasi-rigid operator $T$ on $X$, either $T$ is hyper-recurrent or $\eta(T) = \infty$?
\end{question}

\begin{question}
   Is $\mathrm{Hr}(T)$ always dense in $X$ if $T$ is hyper-recurrent on a separable Fréchet space $X$?
\end{question}

\subsection*{Acknowledgements}

The first author was partially supported by CAPES and CNPq-Brazil. The second author acknowledges financial support from CNPq-Brazil and the Fundação Carlos Chagas Filho de Amparo à Pesquisa do Estado do Rio de Janeiro (FAPERJ) through grant E-26/210.388/2019.


\end{document}